\documentclass{cmslatex}
\usepackage[paperwidth=7in, paperheight=10in, margin=.875in]{geometry}
 \usepackage[backref,colorlinks,linkcolor=red,anchorcolor=green,citecolor=blue]{hyperref}
\usepackage{amsfonts,amssymb}
\usepackage{amsmath}
\usepackage{graphicx}
\usepackage{epstopdf}
\usepackage{cite}
\usepackage{enumerate}
\renewcommand{\theequation}{\arabic{section}.\arabic{equation}}

\usepackage{tikz}

\newcommand{\eps}{\varepsilon} 
 
\newcommand{\RR}{{\mathbb R}} 
\newcommand{\EE}{{\mathbb E}}

\newcommand{\NN}{{\mathbb N}}
\newcommand{\CC}{{\mathbb C}}
\DeclareMathAlphabet{\itbf}{OML}{cmm}{b}{it}

\newcommand{\qed}{\hfill $\Box$ \medskip}

\begin{document}

\title{Intensity fluctuations in random waveguides}

\author{Josselin Garnier\thanks{CMAP, CNRS, Ecole polytechnique, Institut Polytechnique de Paris, 91128 Palaiseau Cedex - France (josselin.garnier@polytechnique.edu). http://www.josselin-garnier.org}}
 
\pagestyle{myheadings} \markboth{Intensity fluctuations in random waveguides}{J. Garnier} 

\maketitle

\begin{abstract}
An asymptotic analysis of wave propagation in randomly perturbed waveguides is carried out in order to identify the effective Markovian dynamics of the guided mode powers.
The main result  consists in a quantification of the fluctuations of the mode powers and wave intensities 
that increase exponentially with the propagation distance.
The exponential growth rate is studied in detail so as to determine its dependence with respect to the waveguide geometry, the statistics of the random perturbations, and the operating wavelength.
\end{abstract}


\begin{AMS} 
35R60; % 	Partial differential equations with randomness, stochastic partial differential equations
35L05;   % Wave equation
60F05; %Central limit and other weak theorems
35Q60; % PDEs in connection with optics and electromagnetic theory
35Q99. % None of the above, but in this section
\end{AMS}

\section{Introduction}
We consider wave propagation in randomly perturbed waveguides.
The random perturbations may affect the index of refraction within the core of the waveguide or the geometry of the core boundary.
An asymptotic analysis based on a separation of scales technique can be applied when the amplitude of the random perturbation
is small, its correlation length is of the same order as the operating wavelength, and the propagation distance is large so that
the net effect of the perturbations is of order one.
The overall result is that the scalar wavefield can be expanded on the complete basis
of the modes of the unperturbed waveguide, that contains guided modes, radiating modes and evanescent modes,
and the complex mode amplitudes of this decomposition follow an effective Markovian dynamics.
In particular the guided mode powers form a Markovian process with a generator that describes random exchange of powers
between the guided modes and power leakage (towards the radiating modes) that can be expressed as a deterministic 
mode-dependent dissipation.
These results can be found in different forms in the physics literature \cite{marcuse,dozier,colosi09} and in the mathematics literature \cite{kohler77,GP07,gomez}.
In this paper we present a unified framework that deals with interior and boundary random fluctuations, 
we clarify the relationships between the mode-dependent dissipation coefficients and the statistics of the random perturbations, and we give a precise characterization of the mode power fluctuations, which is the main original result of the paper and which can be summarized as follows.

The effective Markovian description of the guided mode powers makes it possible to analyze their first- and second-order moments
(that are second- and fourth-order moments of the mode amplitudes),
which in turn gives a statistical description of the intensity distribution of the wavefield.
We find that the relative fluctuations of the intensity are, in general,  characterized by an exponential growth with the propagation distance, whose rate can be defined
as the difference of the first eigenvalues of two symmetric matrices (or two self-adjoint operators).
When the effective dissipation is negligible, we recover the well-known equipartition result  \cite{FGPSbook,GP07}: The exponential growth rate is zero and
the power becomes equipartitioned amongst the guided modes.
When there is effective dissipation, the exponential growth rate can be positive, which means that power fluctuations may become very large, 
as first noticed in the physics literature by Creamer \cite{creamer}.
In fact we show that  the exponential growth rate  is positive as soon as two effective mode-dependent 
dissipation coefficients are different. Our analysis shows that the growth rate increases when the effective mode-dependent 
dissipation coefficients become more different, and it decreases when the number of guided modes increases.
Finally, we analyze a special regime, the continuum approximation, 
in which the operating frequency is large so that the number of guided modes becomes large.
Under such circumstances, we find that the exponential growth rate vanishes.
The exponential growth of the intensity fluctuations can, therefore, only be observed when there is a limited 
number of guided modes, and we recover the standard result that, in open random medium, the wavefield
behaves like a Gaussian-distributed complex field for large propagation distances and the scintillation index that measures the
relative intensity fluctuations becomes equal to one.

The paper is organized as follows.
In Section \ref{sec:intro} we formulate the problem and present the waveguide geometry.
In Section \ref{sec:homo} we review the spectral analysis of the ideal waveguide, when the medium inside the core is homogeneous
and the boundaries are straight.
In Section \ref{sec:random} we explain that the wavefield in the random waveguide can be expanded on the set of
eigenmodes of the ideal waveguide and we identify the set of coupled equations satisfied by the mode amplitudes.
In Section \ref{sec:effmarkov1} we present the effective Markovian dynamics for the mode amplitudes and
in Section \ref{sec:effmarkov2} we remark that the mode powers also satisfy Markovian dynamics.
The long-range behavior of the mean mode powers is described in Section \ref{sec:effmarkov3},
and the fluctuation analysis in Section \ref{sec:fluctuationanalysis} reveals that the normalized variance of the intensity 
grows exponentially with the propagation distance.

\section{Wave propagation in waveguides}
\label{sec:intro}%
Our model consists of a two-dimensional waveguide with range axis denoted by $z \in \RR$ and transverse coordinate
denoted by $x \in \RR$ (see Figure \ref{fig:1}). This may model a dielectric slab waveguide for instance. 
A point-like source at a fixed position $(x,z) = (x_{\rm s},0)$ transmits a time-harmonic signal.
The wavefield ${p}(x,z)$ satisfies the Helmholtz equation:
\begin{align}
\Big[  (\partial_x^2 + \partial_z^2 ) +  k^2 {\rm n}^2 (x,z)   \Big] 
{p}(x,z)  = \delta(z)\delta(x-x_{\rm s}),  \label{eq:pressure0}
\end{align}
for $ (x,z) \in \RR^2$, where $k$ is the homogeneous wavenumber and  ${\rm n}(x,z)$ is the index of refraction at position $(x,z)$.  

In the case of ideal (unperturbed) waveguides, the index of refraction
is range-independent and equal to
\begin{align}
{\rm n}^{(0)}(x)^2 = \left\{
\begin{array}{ll}
{\it{n}}^2 & \mbox{ if } x \in (-d/2,d/2),\\
1 &\mbox{ otherwise,}
\end{array}
\right.
\label{eq:no}
\end{align}
where $n>1$ is the relative index of the core and $d>0$ is its diameter.

We are interested in randomly perturbed waveguides.
In this paper we address two types of random waveguides.

Type I perturbation: in the first type, the index of refraction within the core region $x\in (-d/2,d/2)$ is randomly perturbed \cite{beran,colosi09,colosi12,gomez,kohler77}:
\begin{equation}
\label{eq:modelpert1}
{\rm n}^{(\eps)}(x,z)^2 = \left\{
\begin{array}{ll}
{\it{n}}^2 +\eps \nu(x,z) & \mbox{ if } x \in (-d/2,d/2) \mbox{ and } z\in (0,L^{(\eps)}),\\
1 &\mbox{ otherwise}.
\end{array}
\right.
\end{equation}
The fluctuations are modeled by the zero-mean, bounded,
stationary in $z$ random process $\nu(x,z)$
with smooth covariance function
\begin{equation}
{\cal R}_{\rm I}(x,x',z') = \EE[\nu(x,z)\nu(x',z+z')]. 
\end{equation}
It satisfies strong mixing conditions in $z$ as defined for example in
\cite[section 2]{papa74}. 
The typical amplitude of the fluctuations of index of refraction is assumed to be 
much smaller than $1$ and it is modeled
 by the small and positive dimensionless parameter $\eps$.

\begin{figure}
\centerline{
\hspace*{-0.4cm}
\includegraphics[width=7.8cm]{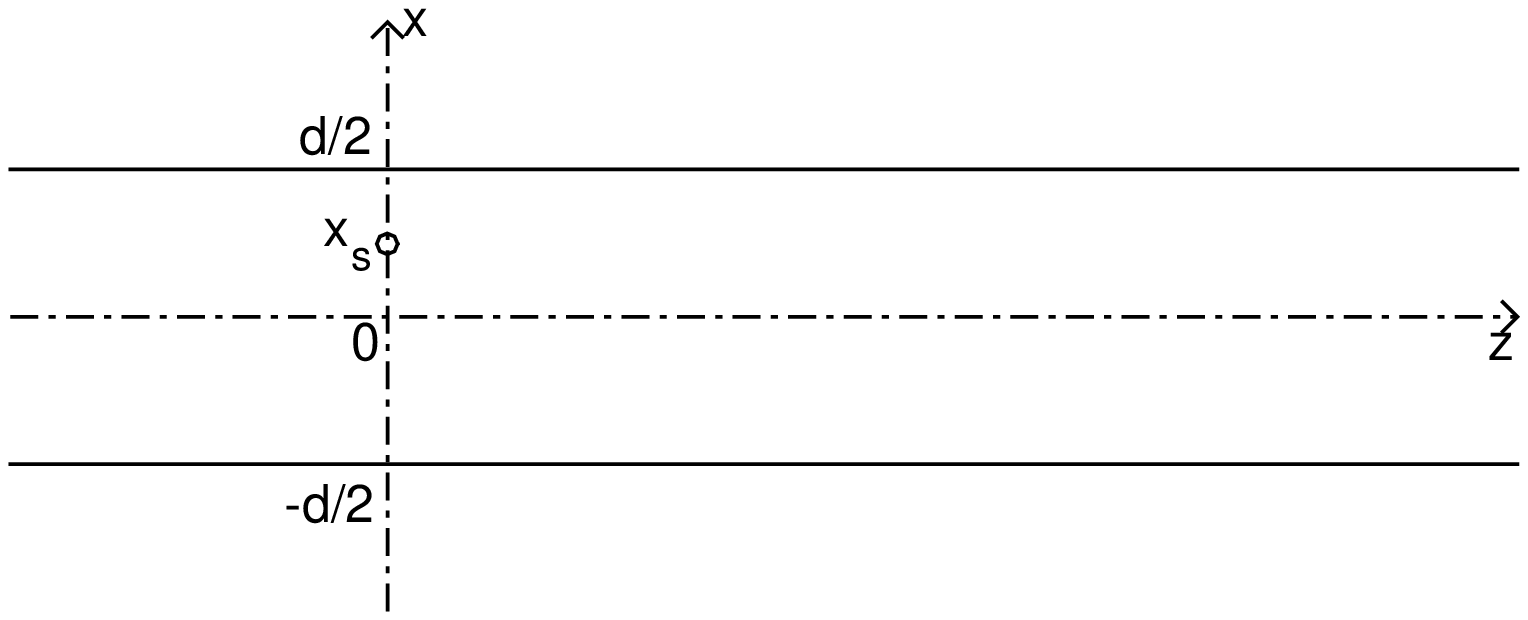}
\hspace*{-1.0cm}
\includegraphics[width=7.8cm]{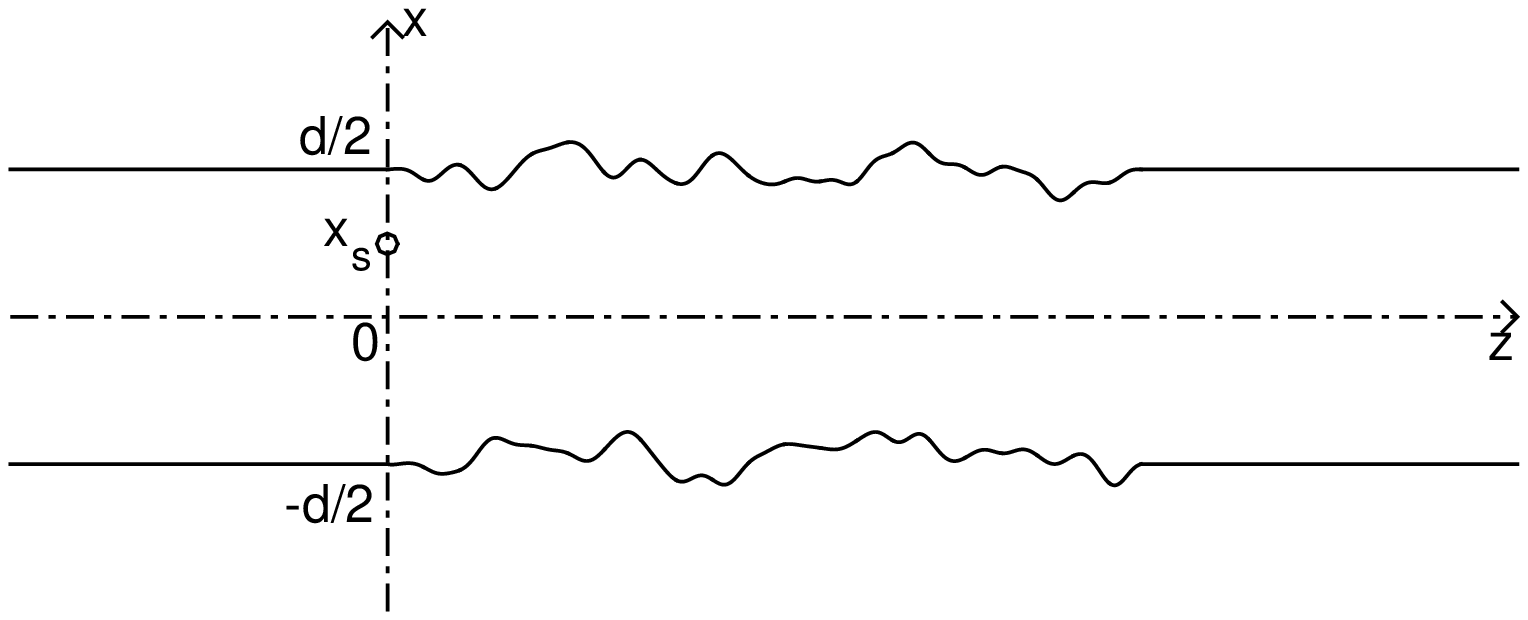}
}
\caption{Left: An ideal two-dimensional waveguide. Right: A two-dimensional waveguide with cross-section perturbed by random fluctuations of the top and bottom boundaries. The point source is in the plane $z=0$. }
\label{fig:1}
\end{figure}

Type II perturbation: in the second type (see Figure \ref{fig:1}), the boundaries of the core are randomly perturbed \cite{alonso,borcea15,marcuse69,marcuse,gomez2}:
\begin{equation}
\label{eq:modelpert2}
{\rm n}^{(\eps)}(x,z)^2 = \left\{
\begin{array}{ll}
n^2 & \mbox{ if } x \in \big( {\cal D}_{-}^{(\eps)}(z), {\cal D}_{+}^{(\eps)}(z) \big)  \mbox{ and } z\in (0,L^{(\eps)}) , \\
 1 &\mbox{ otherwise,}
\end{array}
\right.
\end{equation}
where
\begin{align}
\label{eq:Interfaces1a}
{\cal D}_{-}^{(\eps)}(z) &= -d/2+ \eps d\nu_1(z)  ,  \\
{\cal D}_{+}^{(\eps)}(z) &= d/2  +\eps d\nu_2(z) .
\label{eq:Interfaces1b}
\end{align}
The fluctuations are modeled by the zero-mean, bounded,
independent and identically distributed stationary random processes $\nu_1$ and $\nu_2$ 
with smooth covariance function
\begin{equation}
{\cal R}_{\rm II}(z') = \EE[\nu_q(z)\nu_q(z+z')], \qquad q = 1, 2. \label{eq:covar}
\end{equation}
They satisfy strong mixing conditions.
The typical amplitude of the fluctuations of the boundaries is assumed to be 
much smaller than the core diameter 
$d$ and it is modeled in (\ref{eq:Interfaces1a}-\ref{eq:Interfaces1b}) by the
small and positive dimensionless parameter $\eps$.

We study the wavefield at $z > 0$, satisfying  
\begin{equation}
{p}(x,z)
\in {\cal C}^0\big( (0,+\infty), H^2(\RR)\big)
\cap  {\cal C}^2\big( (0,+\infty),L^2(\RR)\big) ,  
\label{eq:functclass}
\end{equation}
and to set radiation conditions, we have assumed that the random
fluctuations are supported in the range interval $(0,L^{(\eps)})$.
We will see that net scattering effect of these fluctuations  becomes of
order one at range distances of order $\eps^{-2}$, so we consider the interesting case $L^{(\eps)} =
L/\eps^2$. 

\section{Homogeneous waveguide}
\label{sec:homo}%
In this section, we consider an index of refraction of the form (\ref{eq:no}), which is stepwise constant.
There is  no fluctuation of the medium along the $z$-axis.
The analysis of the perfect waveguide is classical \cite{magnanini,wilcox}, we only give the main results. 
The Helmholtz operator has a spectrum of the form
\begin{equation}
\label{eq:spectrum0}
(-\infty,  {k}^2) \cup \{ \beta_{N-1}^2,\ldots,\beta_0^2\}  ,
\end{equation}
where the $N$ modal wavenumbers $\beta_j$  are positive and 
$k^2<\beta_{N-1}^2 < \cdots <\beta_0^2 <n^2 k^2$.
The  generalized eigenfunctions $\phi_{t,\gamma}$, $t\in \{e,o\}$, associated to the spectral parameter $\gamma$ in the continuous spectrum $(-\infty,  {k}^2)$
and the eigenfunctions $\phi_j$, $j=0,\ldots,N-1$,  associated to the discrete spectrum,
are given in Appendix \ref{app:dec}.
The generalized eigenfunctions $\phi_{e,\gamma}$ are even and $\phi_{o,\gamma}$ are odd.
The eigenfunctions $\phi_j$ are even for even $j$ and odd for odd $j$.
Any function can be expanded on the complete set of the eigenfunctions of the Helmholtz operator.
In particular, any solution of the Helmholtz equation in homogeneous medium can be expanded as 
\begin{equation}
\label{eq:modalexpansion}
{p}(x,z) = \sum_{j=0}^{N-1} {p}_j(z) \phi_j(x)
+\sum_{t\in \{e,o\}} \int_{-\infty}^{{k}^2} {p}_{t,\gamma} (z) \phi_{t,\gamma}(x) d\gamma
 .
\end{equation}
The modes for $j=0,\ldots,N-1$ are guided, the modes for $\gamma \in (0,{k}^2)$ are radiating, the modes for 
$\gamma \in (-\infty,0)$ are evanescent.
Indeed, the complex mode amplitudes satisfy
\begin{align}
\partial_z^2 {p}_j +\beta_j^2 {p}_j&=0, \quad  j=0,\ldots,N-1,\\
\partial_z^2 {p}_{t,\gamma} +\gamma {p}_{t,\gamma}&=0, \quad  \gamma \in (-\infty,{k}^2) ,
\end{align}
for any $z\neq 0$.
Therefore, if the source is of the form (\ref{eq:pressure0}), we have for $z>0$:
\begin{align}
\nonumber
{p}(x,z) =&  \sum_{j=0}^{N-1} \frac{{a}_{j,{\rm s}}}{\sqrt{\beta_j}} e^{i \beta_j z} \phi_j(x)
+\sum_{t\in \{e,o\}} \int_0^{{k}^2}  \frac{{a}_{t,\gamma,{\rm s}}}{\gamma^{1/4}} e^{i \sqrt{\gamma} z}\phi_{t,\gamma}(x) d\gamma
\\
&+\sum_{t\in \{e,o\}}\int_{-\infty}^0  \frac{{a}_{t,\gamma,{\rm s}}}{|\gamma|^{1/4}}   e^{- \sqrt{|\gamma|} z} \phi_{t,\gamma}(x) d\gamma
,
\end{align}
where the mode amplitudes are constant and determined by the source:
\begin{align}
{a}_{j,{\rm s}} =  & \frac{\sqrt{\beta_j}}{2} \phi_j(x_{\rm s})  ,\quad j=0,\ldots,N-1,
\label{eq:defajs}
\\
%{a}_{\gamma,{\rm s}} = & \frac{\gamma^{1/4}}{2}   \phi_\gamma(z_0),\quad \gamma \in (0,{k}^2) ,\\
%{a}_{\gamma,{\rm s}} = & \frac{|\gamma|^{1/4}}{2}  \phi_\gamma(z_0) ,\quad \gamma \in (-\infty,0) .
{a}_{t,\gamma,{\rm s}} = & \frac{|\gamma|^{1/4}}{2}  \phi_{t,\gamma}(x_{\rm s}) ,\quad \gamma \in (-\infty,{k}^2) ,\quad t\in \{e,o\}.
\label{eq:defagammas}
\end{align}

\section{Random  waveguide}
\label{sec:random}
We consider the two types of random perturbations described in Section \ref{sec:intro}. In both cases we can write
$$
{\rm n}^2(x,z) = {\rm n}^{(0)}(x)^2  + V^{(\eps)}(x,z)  {\bf 1}_{(0,L^{(\eps)})}(z) ,
$$
where the perturbation is of the form
\begin{align}
\label{eq:newindrefI}
V^{(\eps)} (x,z) = \eps  \nu(x,z)
\end{align}
for type I perturbations, and 
\begin{align}
\nonumber
V^{(\eps)}(x,z) =&
 (n^2-1) 
\big[ -
 {\bf 1}_{(-d/2,-d/2+\eps d \nu_1(z))}(x) {\bf 1}_{(0,+\infty)}(\nu_1(z))  \\
\nonumber
 &
 \quad +
 {\bf 1}_{(-d/2+\eps d \nu_1(z),-d/2)}(x) {\bf 1}_{(-\infty,0)}(\nu_1(z)) 
 \big]  \\
\nonumber
 &+   (n^2-1)  
\big[
 {\bf 1}_{(d/2, d/2+\eps d \nu_2(z))}(x) {\bf 1}_{(0,+\infty)}(\nu_2(z))\\
  &
\quad -
 {\bf 1}_{(d/2+\eps d \nu_2(z),d/2)}(x) {\bf 1}_{(-\infty,0)}(\nu_2(z)) 
 \big] 
\label{eq:newindref}
\end{align}
for type II perturbations.

The solution of the perturbed Helmholtz equation (\ref{eq:pressure0})
can be expanded as (\ref{eq:modalexpansion}) and the complex mode amplitudes satisfy the coupled equations
for $z\in (0,L^{(\eps)})$:
\begin{align}
\label{eq:cma2a}
\partial_z^2 {p}_j +\beta_j^2 {p}_j&= -k^2 \sum_{l=0}^{N-1} C^{(\eps)}_{j,l}(z) {p}_l 
-k^2 \sum_{t'\in \{e,o\}} \int_{-\infty}^{{k}^2} C^{(\eps)}_{j,t',\gamma'} (z) {p}_{t',\gamma'} d\gamma',
\end{align}
for $j=0,\ldots,N-1$,
\begin{align}
\partial_z^2 {p}_{t,\gamma} +\gamma {p}_{t,\gamma}&= -k^2 \sum_{l=0}^{N-1} C^{(\eps)}_{t,\gamma , l}(z) {p}_l 
- k^2 \sum_{t'\in \{e,o\}}\int_{-\infty}^{{k}^2} C^{(\eps)}_{t,\gamma,t',\gamma'} (z) {p}_{t',\gamma'} d\gamma',  
\label{eq:cma2b}
\end{align}
for $ \gamma \in (-\infty,{k}^2) $ and $t \in \{e,o\}$, with
\begin{align}
C^{(\eps)}_{j,l}(z) =& \left(\phi_j,\phi_l V^{(\eps)}(\cdot,z) \right)_{L^2} ,\\
C^{(\eps)}_{j,t',\gamma'}(z) =& \left(\phi_j,\phi_{t',\gamma'} V^{(\eps)}(\cdot,z)    \right)_{L^2} ,\\
C^{(\eps)}_{t,\gamma , l}(z) =&  \left(\phi_{t,\gamma},\phi_l V^{(\eps)}(\cdot,z)  \right)_{L^2} ,\\
C^{(\eps)}_{t,\gamma,t',\gamma'}(z) =& \left(\phi_{t,\gamma},\phi_{t',\gamma'} V^{(\eps)}(\cdot,z)   \right)_{L^2} ,
\end{align}
and $\left(\cdot,\cdot\right)_{L^2}$ stands for the standard scalar product in $L^2(\RR)$ (see (\ref{scalarproduct})).
These equations are obtained by substituting the ansatz (\ref{eq:modalexpansion}) into (\ref{eq:pressure0}) and by projecting onto the eigenmodes.

From the definitions \eqref{eq:newindrefI} or (\ref{eq:newindref}) of $V^{(\eps)}(x,z)$ and the Taylor expansions of the eigenfunctions $\phi_{j}(x)$ and $\phi_{t,\gamma}(x)$ 
around $x = \pm d/2$, we obtain power series (in $\eps$) expressions of the coefficients $C^{(\eps)}_{j,l}$:
\begin{align}
C_{j, l}^{(\eps)} (z) =& \eps C_{j,l}(z)  + \eps^2 c_{j,l}(z)  +o(\eps^2) ,\\
C_{j, l}(z) =& 
\left\{
\begin{array}{ll}
\left(\phi_j,\phi_l \nu(\cdot,z) \right)_{L^2}  &\mbox{ type I}\\
 (n^2-1) d \big\{ - \nu_1(z) [\phi_{j} \phi_{l}] \big(-\frac{d}{2}\big) +  \nu_2(z)[\phi_{j} \phi_{l}] 
\big(\frac{d}{2}\big)  \big\}
&\mbox{ type II}
\end{array}
\right.
\label{expres:Cttp} ,\\
c_{j,l}(z) =& 
\left\{
\begin{array}{ll}
0 &\mbox{ type I}\\
\frac{(n^2-1)d^2}{2} \big\{ - \nu_1^2(z)  \partial_x[ \phi_{j} \phi_{l}] \big(-\frac{d}{2}\big) +  \nu_2^2(z)  \partial_x[\phi_{j} \phi_{l}] \big(\frac{d}{2}\big)  \big\}
&\mbox{ type II}
\end{array}
\right.
,
 \label{eq:csmall}
\end{align}
and similarly for  $C_{j ,t, \gamma}^{(\eps)}$, $C_{t,\gamma, l}^{(\eps)}$, and $C_{t,\gamma,t', \gamma'}^{(\eps)}$.

We finally introduce the generalized forward-going and backward-going mode amplitudes:
\begin{equation}
\label{eq:amplitudes}
\{a_{j}(z), \, b_{j}(z), ~ j = 0, \ldots, N-1\} ~~ \mbox{and} ~~ \{a_{t,\gamma}(z), \, b_{t,\gamma}(z), ~ \gamma \in (0,k^2)\},
\end{equation}
for $t \in \{e,o\}$, which are defined such that 
\begin{align}
{p}_{j}(z) =& \frac{1}{\sqrt{\beta_{j}}}\Big( {a}_{j} (z) e^{i\beta_{j} z} +{b}_{j}(z) e^{- i\beta_{j} z} \Big), \nonumber \\
\partial_z  {p}_{j}(z) =& i\sqrt{\beta_{j}}\Big( {a}_{j} (z) e^{i\beta_{j}z} -{b}_{j}(z) e^{- i\beta_{j} z} \Big),\quad j=0,\ldots,N-1,
\label{eq:guidedFB}
\end{align}
and 
\begin{align}
{p}_{t,\gamma}(z) =& \frac{1}{\gamma^{1/4}}\Big( {a}_{t,\gamma} (z) e^{i\sqrt{\gamma} z} 
+{b}_{t,\gamma}(z) e^{- i\sqrt{\gamma}z} \Big), \nonumber \\
\partial_z  {p}_{t,\gamma}(z) =& i\gamma^{1/4}\Big( {a}_{t,\gamma} (z) e^{i\sqrt{\gamma} z} 
-{b}_{t,\gamma}(z) e^{- i\sqrt{\gamma}z} \Big),\quad \gamma\in (0,k^2) ,\quad t \in \{e,o\}.
\label{eq:radFB}
\end{align}
We can substitute (\ref{eq:guidedFB}--\ref{eq:radFB}) into (\ref{eq:cma2a}--\ref{eq:cma2b}) in order to
obtain the  first-order system of coupled random differential equations 
satisfied by the mode amplitudes \eqref{eq:amplitudes}:
\begin{align}
\nonumber
\partial_z {a}_{j}(z) =& \frac{i k^2}{2} \hspace{-0.02in}
\sum_{l'=0}^{N-1} \frac{C_{j,l'}^{(\eps)}(z)}{\sqrt{\beta_{l'}\beta_{j}}} \Big[ {a}_{l'}(z) e^{i (\beta_{l'}-\beta_{j})z}
+ {b}_{l'}(z) e^{i (-\beta_{l'}-\beta_{j})z}
\Big]\\
\nonumber
&\hspace{-0.4in}+  \frac{i  k^2}{2} \hspace{-0.1in} \sum_{t' \in \{e,o\}} 
\int_0^{k^2}  \frac{C_{j,t',\gamma'}^{(\eps)}(z)}{\sqrt[4]{\gamma'}\sqrt{\beta_{j}}} \Big[ {a}_{t',\gamma'}(z) e^{i (\sqrt{\gamma'}-\beta_{j})z}
+ {b}_{t',\gamma'}(z) e^{i (-\sqrt{\gamma'}-\beta_{j})z}
\Big]d\gamma'\\
&\hspace{-0.4in}+
 \frac{i k^2}{2} \hspace{-0.1in}\sum_{t' \in \{e,o\}}
\int_{-\infty}^0 \frac{C^{(\eps)}_{j,t',\gamma'}(z)}{\sqrt{\beta_{j}}}  {p}_{t',\gamma'}(z) e^{-i \beta_{j}z}
d\gamma' ,
\label{eq:evol1a}
\end{align}
\begin{align}
\nonumber
\partial_z {a}_{t,\gamma}(z) =& \frac{i k^2}{2} \hspace{-0.02in} 
\sum_{l'=0}^{N-1} \frac{C_{t,\gamma,l'}^{(\eps)}(z)}{\sqrt[4]{\gamma} \sqrt{\beta_{l'}}} \Big[ {a}_{l'}(z) e^{i (\beta_{l'}-\sqrt{\gamma})z}
+ {b}_{l'}(z) e^{i (-\beta_{l'}-\sqrt{\gamma})z}
\Big]\\
\nonumber
&\hspace{-0.4in}+  \frac{i  k^2}{2} \hspace{-0.1in}\sum_{t' \in \{e,o\}} 
\int_0^{k^2}  \frac{C_{t,\gamma,t',\gamma'}^{(\eps)}(z)}{\sqrt[4]{{\gamma'}{\gamma}}} \Big[ {a}_{t',\gamma'}(z) e^{i (\sqrt{\gamma'}-\sqrt{\gamma})z}
+ {b}_{t',\gamma'}(z) e^{i (-\sqrt{\gamma'}-\sqrt{\gamma})z}
\Big]d\gamma'\\
&\hspace{-0.4in}+
 \frac{i  k^2}{2} \hspace{-0.1in} \sum_{t' \in \{e,o\}}
\int_{-\infty}^0 \frac{C_{t,\gamma,t',\gamma'}^{(\eps)}(z)}{\sqrt[4]{\gamma}}  {p}_{t',\gamma'}(z) e^{-i \sqrt{\gamma}z}
d\gamma' ,
\label{eq:evol1b}
\end{align}
with similar equations for $b_j$ and $b_{t,\gamma}$.
This system is complemented with the boundary conditions at $z=0$ and $z=L^{(\eps)}$:
$$
{a}_{j} (0) = a_{j,{\rm s}}, \quad {b}_{j} (L^{(\eps)}) = 0, \quad {a}_{t,\gamma} (0) = a_{t,\gamma,{\rm s}},
\quad {b}_{t,\gamma} (L^{(\eps)}) = 0,
$$
where ${a}_{j,{\rm s}} $ and ${a}_{t,\gamma,{\rm s}}$ are defined by (\ref{eq:defajs}-\ref{eq:defagammas}).
The evanescent mode amplitudes ${p}_{t,\gamma}$, $t\in \{e,o\}$, $\gamma\in (-\infty,0)$, 
satisfy (\ref{eq:cma2b}).

\section{The effective Markovian dynamics for the mode amplitudes}
\label{sec:effmarkov1}
We rename the complex mode amplitudes in the long-range scaling as 
\begin{align}
\label{eq:rescAmplitudesa}
&
{a}_{j}^{{\eps}}(z) = {a}_{j}\Big(\frac{z}{\eps^2}\Big),\quad
{b}_{j}^{{\eps}}(z) = {b}_{j}\Big(\frac{z}{\eps^2}\Big),\, \quad j=0,\ldots,N-1 ,  \\
&{a}_{t,\gamma}^{{\eps}}(z) = {a}_{t,\gamma}\Big(\frac{z}{\eps^2}\Big),~~\,
{b}_{t,\gamma}^{{\eps}}(z) = {b}_{t,\gamma}\Big(\frac{z}{\eps^2}\Big), \quad \gamma\in (0,k^2),\quad t\in \{e,o\} .
\label{eq:rescAmplitudesb}
\end{align}
We can follow the lines of  \cite{gomez} to get the following results.

1) In the regime $\eps \ll 1$ the evanescent mode amplitudes, that satisfy (\ref{eq:cma2b}),
can be expressed to leading order in closed forms as functions of the 
guided and radiating mode amplitudes (\ref{eq:rescAmplitudesa}-\ref{eq:rescAmplitudesb}).
Indeed it is possible to invert the  operator $\partial_z^2 + \gamma$ in (\ref{eq:cma2b}) for $\gamma<0$ by using the  Green's function that satisfies 
the radiation condition and to obtain:
\begin{align}
& {p}_{t,\gamma}(\frac{z}{\eps^2}) =  
\frac{\eps  k^2}{2\sqrt{|\gamma|}}   
\int_{0}^{L/\eps^2} 
\sum_{l'=0}^{N-1} \Bigg\{\frac{C_{t,\gamma,l'}(z')}{\sqrt{\beta_{l'}}} \Big[ {a}_{l'}^\eps(z) e^{i \beta_{l'} z'}
+ {b}^\eps_{l'} (z)e^{- i\beta_{l'} z'}
\Big]\nonumber \\
&+\int_0^{k^2}  \frac{C_{t,\gamma,t',\gamma'}(z')}{\sqrt[4]{\gamma'}} \Big[ {a}^\eps_{t',\gamma'}(z) e^{i  \sqrt{\gamma'} z'}
+ {b}^\eps_{t',\gamma'}(z) e^{- i \sqrt{\gamma'} z'}
\Big] d\gamma'  \Bigg\}
e^{-\sqrt{|\gamma|} |\frac{z}{\eps^2}-z'|} dz' 
\nonumber  \\
& + O(\eps^2), 
\end{align}
for $z > 0, \gamma < 0$ and $t \in \{e, o\}$. Here we recognize $G_\gamma (z,z')=\frac{1}{2\sqrt{|\gamma|}} e^{-\sqrt{|\gamma|} |z-z'|} $ that is the Green's function
of the equation $\partial_z^2 G_\gamma (z,z')+\gamma G_\gamma (z,z')= - \delta(z-z')$ for $\gamma<0$.

2) Under the assumption that the power spectral density $\widehat {\cal R}_{\rm I}(\kappa,x,x')$ for type-I perturbations
(or $\widehat {\cal R}_{\rm II}(\kappa)$ for type-II perturbations) 
has compact support or fast decay,
the forward-scattering approximation can be proved, i.e. the coupling between forward-going and backward-going mode amplitudes
is negligible, so that we have
\[
b_{j}^{{\eps}}(z) \approx 0,  \quad j = 0, \ldots, N-1, \qquad b_{t,\gamma}^{{\eps}}(z) \approx 0, ~~ \gamma \in (0,k^2), ~~ 
t \in \{e,o\}.
\]

3) The forward-going guided mode amplitudes $({a}_{j}^\eps)_{j=0}^{N-1}$ and radiating mode amplitudes $({a}_{t,\gamma}^\eps)_{\gamma \in (0,k^2),t\in \{e,o\}}$
then satisfy a closed linear system of the form
$$
\frac{d {\itbf a}^\eps}{dz} = \frac{1}{\eps} {\bf F}( \frac{z}{\eps^2}) {\itbf a}^\eps + {\bf G}( \frac{z}{\eps^2}) {\itbf a}^\eps +o(1),
$$
with initial conditions for ${\itbf a}^\eps$ at $z=0$.
Here $ {\bf F}$, resp. ${\bf G}$, is an operator with zero mean, resp. non-zero mean, and ergodic properties
inherited from those of the processes $\nu$.

We can finally apply a diffusion approximation theorem to establish the following result (see \cite{gomez} for the full statement or \cite{kohler77} for a first version in which the contributions of the evanescent modes is neglected, which means that the operator ${\cal L}^3$ is missing in the expression of the generator ${\cal L}$).

\begin{proposition}
The  random process 
$$
\big( ({a}_{j}^\eps(z) )_{j=0}^{N-1}, ({a}_{t,\gamma}^\eps(z) )_{\gamma \in (0,{k}^2), t\in \{e,o\}} \big) 
$$
converges in distribution in ${\cal C}^0([0,L], \mathbb{C}^{N}  \times L^2((0,{k}^2))^2 )$, 
the space of continuous functions from $[0,L]$ to $\mathbb{C}^{N}  \times L^2((0,{k}^2))^2$,
to the Markov process  
$$
\big( (\mathfrak{a}_{j}(z) )_{j=0}^{N-1} , (\mathfrak{a}_{t,\gamma} (z))_{\gamma \in (0,{k}^2),t\in \{e,o\}} \big)
$$
with infinitesimal generator
${\cal L}$.
Here $\mathbb{C}^{N}  \times L^2((0,{k}^2))^2$
is equipped with the weak topology and the infinitesimal generator has the form
$
{\cal L}= {\cal L}^1+{\cal L}^2+{\cal L}^3 ,
$
where ${\cal L}^j$, $1 \leq j \leq 3$, are the differential operators:
\begin{align}
\nonumber {\cal L}^1 = & \frac{1}{2} 
\sum_{j,l=0}^{N-1} 
\Gamma_{jl} \big( 
\mathfrak{a}_{j} \overline{\mathfrak{a}_{j}} \partial_{\mathfrak{a}_{l}} \partial_{\overline{\mathfrak{a}_{l}}}
+
\mathfrak{a}_{l} \overline{\mathfrak{a}_{l}} \partial_{\mathfrak{a}_{j}} \partial_{\overline{\mathfrak{a}_{j}}}
% \\ \nonumber &\quad \quad
-
\mathfrak{a}_{j}  \mathfrak{a}_{l} \partial_{\mathfrak{a}_{j}} \partial_{ \mathfrak{a}_{l}}
-
\overline{\mathfrak{a}_{j}}  \overline{\mathfrak{a}_{l}} \partial_{\overline{\mathfrak{a}_{j}}} \partial_{ \overline{\mathfrak{a}_{l}}}\big)
{\bf 1}_{ j\neq l}
\\
\nonumber &+
 \frac{1}{2} 
\sum_{j,l=0}^{N-1} 
\Gamma^1_{j l} \big( 
\mathfrak{a}_{j} \overline{\mathfrak{a}_{l}} \partial_{\mathfrak{a}_{j}} \partial_{\overline{\mathfrak{a}_{l}}}
+
\overline{\mathfrak{a}_{j}} \mathfrak{a}_{l} \partial_{\overline{\mathfrak{a}_{j}}} \partial_{\mathfrak{a}_{l}}
%\\   \nonumber  &\quad \quad
-
\mathfrak{a}_{j}  \mathfrak{a}_{l} \partial_{\mathfrak{a}_{j}} \partial_{ \mathfrak{a}_{l}}
-
\overline{\mathfrak{a}_{j}}  \overline{\mathfrak{a}_{l}} \partial_{\overline{\mathfrak{a}_{j}}} \partial_{ \overline{\mathfrak{a}_{l}}}\big)
\\ 
%\nonumber 
&+ \frac{1}{2} 
\sum_{j=0}^{N-1} \big( \Gamma_{jj} - \Gamma^1_{jj}\big)
\big( \mathfrak{a}_{j} \partial_{\mathfrak{a}_{j}} + \overline{\mathfrak{a}_{j}} \partial_{\overline{\mathfrak{a}_{j}}}
\big)
%\\  &
+\frac{i}{2} 
\sum_{j=0}^{N-1}  \Gamma^{s}_{jj}  
\big( \mathfrak{a}_{j} \partial_{\mathfrak{a}_{j}} - \overline{\mathfrak{a}_{j}} \partial_{\overline{\mathfrak{a}_{j}}}
\big)  , \label{eq:defL1}
\\
{\cal L}^2
=&
-\frac{1}{2}  
\sum_{j=0}^{N-1} ( \Lambda_{j}  +i\Lambda^{s}_{j} ) \mathfrak{a}_{j} \partial_{\mathfrak{a}_{j}} 
+
( \Lambda_{j}  - i\Lambda^{s}_{j} )  \overline{\mathfrak{a}_{j}} \partial_{\overline{\mathfrak{a}_{j}}}  ,
 \label{eq:defL2}\\{\cal L}^3
=&
i  
\sum_{j=0}^{N-1}  \kappa_{j}   \big( \mathfrak{a}_{j} \partial_{\mathfrak{a}_{j}} 
- \overline{\mathfrak{a}_{j}} \partial_{\overline{\mathfrak{a}_{j}}} \big) .
 \label{eq:defL3}
\end{align}
In these definitions we use the  classical complex derivative:
if $ \zeta=\zeta_r+i\zeta_i$, then $\partial_\zeta=(1/2)(\partial_{\zeta_r}-i \partial_{\zeta_i})$ and
$\partial_{\overline{\zeta}} =(1/2)(\partial_{\zeta_r} +i \partial_{\zeta_i})$,
and the coefficients of the operators (\ref{eq:defL1}-\ref{eq:defL3}) are defined for 
$j ,l= 0, \ldots, N-1$, as follows: 

 \noindent - For all $j \neq l$, $\Gamma_{jl}$ and $\Gamma^{s}_{j l}$ are given by 
\begin{align}
\Gamma_{jl}=& \frac{k^4}{2 \beta_j \beta_l}
\int_0^\infty{\cal R}_{jl}(z) \cos \big((\beta_l-\beta_j)z\big) dz , 
\label{def:Gammalj}
\\
\Gamma^{s}_{jl} =&
\frac{k^4}{2 \beta_j\beta_l}
\int_0^\infty 
{\cal R}_{jl}(z)  \sin \big( (\beta_l-\beta_j)z \big) dz ,
\end{align}
with ${\cal R}_{jl}(z) $ defined by
\begin{align}
\label{def:calRjl}
{\cal R}_{jl}(z) &:= \EE [C_{j,l}(0)C_{j,l}(z)] ,\\
\EE [C_{j,l}(0)C_{j',l'}(z)] &= 
\left\{
\begin{array}{ll}
\int_{-d/2}^{d/2} \int_{-d/2}^{d/2}\phi_j \phi_l(x) {\cal R}_{\rm I}(x,x',z) \phi_{j'} \phi_{l'}(x') dx dx' 
&\mbox{ type I}\\
 (n^2-1)^2 d^2 \big[   \phi_{j} \phi_{l} \phi_{j'} \phi_{l'}
\big(-\frac{d}{2}\big) \\
\quad \quad \quad  \quad \quad + \phi_{j} \phi_{l} \phi_{j'} \phi_{l'}
\big(\frac{d}{2}\big) \big] {\cal R}_{\rm II}(z) 
&\mbox{ type II}
\end{array}
\right.
\end{align}
\noindent - For all $j,l$:
\begin{align*}
\Gamma^1_{jl} =&
%\frac{k^4}{2 \beta_j\beta_l} \int_0^\infty \EE\big[ C_{j,j} (0) C_{l,l}(z) \big]   dz.
\frac{k^4}{4 \beta_j\beta_l} \int_0^\infty \EE\big[ C_{j,j} (0) C_{l,l}(z) \big]   +\EE\big[ C_{l,l} (0) C_{j,j}(z) \big]    dz.
\end{align*}
\noindent - For all $j$, $\Lambda_j$ is defined by 
\begin{align}
\label{def:Lambdaj}
\Lambda_j =& \int_0^{{k}^2} \frac{k^4}{2\sqrt{\gamma}\beta_j} \sum_{t\in \{e,o\}}
\int_0^\infty {\cal R}_{j,t,\gamma}(z) \cos \big((\sqrt{\gamma}-\beta_j)z\big) dz d\gamma
\end{align}
and
\begin{align*}
\Gamma_{jj} =& -\hspace{-0.05in}\sum_{l =0,l\neq j}^{N-1} \Gamma_{jl} ,\quad \quad
\Gamma^{s}_{jj} = -\hspace{-0.05in}\sum_{l =0,l\neq j}^{N-1} \Gamma^{s}_{jl}  ,\\
\Lambda_{j}^{s} =&\hspace{-0.05in}  \sum_{t\in \{e,o\}}  \int_0^{{k}^2} \hspace{-0.05in}\frac{k^4}{2\sqrt{\gamma}\beta_j}
\int_0^\infty {\cal R}_{j,t,\gamma}(z) 
\sin \big[ (\sqrt{\gamma}-\beta_j)z\big] dz d\gamma, \\
\kappa_{j} = &  \hspace{-0.05in}  \sum_{t\in \{e,o\}} \int_{-\infty}^0  \frac{k^4}{2\sqrt{|\gamma|}\beta_j}
\int_0^\infty {\cal R}_{j,t,\gamma}(z)  \cos ( \beta_j z) e^{-\sqrt{|\gamma|}z}dz d\gamma + \frac{k^2}{2 \beta_j} \EE[ c_{j,j}(0)],
\end{align*} 
where 
${\cal R}_{j,t,\gamma}(z)= \EE [ C_{j,t,\gamma}(0)C_{j,t,\gamma}(z) ]$ is defined as in (\ref{def:calRjl}) upon substitution $(t,\gamma)$ for $l$
and
$$
\EE[ c_{j,j}(0)]=
\left\{
\begin{array}{ll}
0 &\mbox{ type I}\\
(n^2-1)d^2 {\cal R}_{\rm II}(0) \partial_x [\phi_j^2]\big(\frac{d}{2}\big) &\mbox{ type II}
\end{array}
\right.
$$
\end{proposition}

We give some remarks before focusing our attention on the mode powers.

1) The convergence result holds in the weak topology. This means that we can only compute quantities of the 
form $\EE [ F( \mathfrak{a}_0,\ldots,\mathfrak{a}_{N-1},  \int_0^{{k}^2} \alpha_{e,\gamma} \mathfrak{a}_{e,\gamma} d\gamma,  \int_0^{{k}^2} \alpha_{o,\gamma} \mathfrak{a}_{o,\gamma} d\gamma)]$
for any test functions $\alpha_e,\alpha_o \in L^2((0,{k}^2))$ and $F:\mathbb{R}^{N+2}\to \mathbb{R}$.
These quantities 
 are the limits of $\EE [ F( {a}^\eps_0,\ldots, {a}^\eps_{N-1},  \int_0^{{k}^2} \alpha_{e,\gamma}  {a}^\eps_{e,\gamma} d\gamma,  \int_0^{{k}^2} \alpha_{o,\gamma}  {a}^\eps_{o,\gamma} d\gamma)]$ as $\eps \to 0$. 
 
2) The generator ${\cal L}$ does not involve $\partial_{\mathfrak{a}_{t,\gamma}}$ or  $\partial_{\bar{\mathfrak{a}}_{t,\gamma}}$. Therefore 
$  ( {a}_{j}^\eps (z))_{j=0}^{N-1}   $
converges in distribution in ${\cal C}^0([0,L], \mathbb{C}^{N}  )$ to
the Markov process  $  (\mathfrak{a}_{j}(z) )_{j=0}^{N-1}   $
with  generator ${\cal L}$. The weak and strong topologies are the same  in $\mathbb{C}^N$,
so we can compute any moment of the 
form $\EE [ F( \mathfrak{a}_0,\ldots,\mathfrak{a}_{N-1})]$, which are the limits of $\EE [ F(  {a}^\eps_0,\ldots, {a}^\eps_{N-1})]$.

3) ${\cal L}_1$ is the contribution of the coupling between guided modes, 
which gives rise to power exchange between the guided modes {(effective diffusion)}.

4) ${\cal L}_2$  is the contribution of the coupling between guided and radiating modes, 
which gives rise to power leakage from the guided modes to the radiating ones {(effective attenuation)}
and addition of frequency-dependent phases on the guided mode amplitudes (effective dispersion).
The effective {attenuation} and dispersion are produced by causal phenomena and they
{are} related to each other through Kramers-Konig relations \cite{garniersolna:wrcm}.

5)
${\cal L}_3$  is the contribution of the coupling between guided and evanescent modes,
which gives rise to additional phase terms on the guided mode amplitudes (effective dispersion).
This term is the main effect when the waveguide supports only one propagating mode and the core boundaries 
are hard or soft so that there is no radiating mode \cite{garnier_evan}.

6) If the generator ${\cal L}$ is applied to a test function that depends only on the mode powers $  (P_{j}   )_{j=0}^{N-1}   $,
with $P_{j} = |\mathfrak{a}_{j}|^2$,
then the result is a function that depends only on $ (P_{j} )_{j=0}^{N-1}  $. Thus, 
the mode powers $  (P_{j}(z) )_{j=0}^{N-1}   $ define  a Markov process, with infinitesimal generator defined by (\ref{eq:genLP}) below.

7) The radiation mode amplitudes  remain constant on $L^2((0,{k}^2))^2$, equipped with the weak topology, as $\eps \to 0$. However, this does not
describe the power $ \sum_{t\in \{e,o\}}  \int_0^{{k}^2} |{a}^\eps_{t,\gamma}|^2 d\gamma$ transported by the radiation modes,  because the convergence does not hold in the strong topology of $L^2((0,{k}^2))^2$ so we do not have $\sum_{t\in \{e,o\}} \int_0^{{k}^2} | {a}^\eps_{t,\gamma}|^2 d\gamma \to \sum_{t\in \{e,o\}}  \int_0^{{k}^2} |\mathfrak{a}_{t,\gamma}|^2 d\gamma$ as $\eps \to 0$.

8) When $N=1$, then the generator is
\begin{align}
\nonumber {\cal L}
=&
 \frac{\Gamma^1_{00}}{2}  \big( 
2 \mathfrak{a}_{0} \overline{\mathfrak{a}_{0}} \partial_{\mathfrak{a}_{0}} \partial_{\overline{\mathfrak{a}_{0}}}
-
\mathfrak{a}_{0}  \mathfrak{a}_{0} \partial_{\mathfrak{a}_{0}} \partial_{ \mathfrak{a}_{0}}
-
\overline{\mathfrak{a}_{0}}  \overline{\mathfrak{a}_{0}} \partial_{\overline{\mathfrak{a}_{0}}} \partial_{ \overline{\mathfrak{a}_{0}}}
-
 \mathfrak{a}_{0} \partial_{\mathfrak{a}_{0}} - \overline{\mathfrak{a}_{0}} \partial_{\overline{\mathfrak{a}_{0}}}
\big) \\
& - \frac{\Lambda_{0}  }{2} 
\big( \mathfrak{a}_{0} \partial_{\mathfrak{a}_{0}} + \overline{\mathfrak{a}_{0}} \partial_{\overline{\mathfrak{a}_{0}}} \big)
+\frac{i}{2}  ( \kappa_{0}  - \Lambda^{s}_{0} )
\big( \mathfrak{a}_{0} \partial_{\mathfrak{a}_{0}} - \overline{\mathfrak{a}_{0}} \partial_{\overline{\mathfrak{a}_{0}}}
\big) .
\end{align}
This shows that $\mathfrak{a}_0$ (the amplitude of the unique guided mode) has the same distribution as
$$
\mathfrak{a}_0(z) = a_{0,{\rm s}} \exp\Big( \frac{i}{2}  ( \kappa_{0}  - \Lambda^{s}_{0} ) z+ i 
\sqrt{\Gamma^1_{00}}   W^1_z - \frac{\Lambda_{0}  }{2}  z  \Big) ,
$$
where $W^1_z$ is a standard Brownian motion. The mode amplitude experiences a random phase modulation and a deterministic damping,
which both depend on frequency and two-point statistics of the medium perturbations \cite{garnier_evan}.

9) When $N\geq 2$, the limit process $  (\mathfrak{a}_{j}(z) )_{j=0}^{N-1}   $
can be identified as the solution of a system of stochastic differential equations
driven by Brownian motions.
\begin{corollary}
Let $(W^1_{j})_{j=0}^{N-1}$ be a $N$-dimensional correlated Brownian motion  with covariance function
$$
\EE \big[ W^1_{j,z} W^1_{j',z'}\big] =\Gamma^1_{jj'} z\wedge z' .
$$
Let $(W_{jl})_{0\leq j < l \leq N-1}$ and $(\tilde{W}_{jl})_{0\leq j < l \leq N-1}$ be independent standard Brownian motions.
Set $W_{lj}=W_{jl}$ and $\tilde{W}_{lj}=-\tilde{W}_{jl}$ for $j <l$.\\
Then the limit Markov process $  (\mathfrak{a}_{j}(z) )_{j=0}^{N-1}   $
has the same distribution as the unique solution of
$$
d  \mathfrak{a}_{j} = i  \mathfrak{a}_{j} \circ dW^1_{j,z} +  \sum_{l\neq j} \frac{\sqrt{\Gamma_{jl}}}{\sqrt{2}} \mathfrak{a}_{l} \circ \big(
i d W_{jl,z} - d \tilde{W}_{jl,z} \big) 
 + \frac{1}{2} \big( i\Gamma^s_{jj} -  \Lambda_j - i\Lambda^s_j +i \kappa_j\big)   \mathfrak{a}_{j} dz ,
$$
starting from $\mathfrak{a}_{j}(z=0) =a_{j,{\rm s}}$, $j=0,\ldots,N-1$, or, in It\^o's form:
\begin{align*}
d  \mathfrak{a}_{j} =& i  \mathfrak{a}_{j}   dW^1_{j,z} +  \sum_{l\neq j} \frac{\sqrt{\Gamma_{jl}}}{\sqrt{2}} \mathfrak{a}_{l}\big(
i  d W_{jl,z} -  d \tilde{W}_{jl,z} \big)  \\
&  + \frac{1}{2} \big( \Gamma_{jj} + i\Gamma^s_{jj} -\Gamma^1_{jj} -  \Lambda_j - i\Lambda^s_j +i \kappa_j\big)   \mathfrak{a}_{j} dz .
\end{align*}
\end{corollary}
The proof of the corollary is a straightforward application of It\^o's formula.

\section{The effective Markovian dynamics for the  mode powers}
\label{sec:effmarkov2}
From the result for the complex mode amplitudes we get the following result.
\begin{corollary}
The process $(| {a}_j^\eps(z)|^2)_{j=0}^{N-1}$ 
converges towards a Markov process ${\itbf P}(z)=(P_j(z))_{j=0}^{N-1}$ whose infinitesimal generator $\mathcal{L}_{\itbf P}$ writes:
\begin{align}
\mathcal{L}_{\itbf P} =& \sum_{j \neq l}  \Gamma_{jl} \Big[P_l P_j (\frac{\partial}{\partial P_j} - \frac{\partial}{\partial P_l}) \frac{\partial}{\partial P_j} + (P_l - P_j) \frac{\partial}{\partial P_j}\Big] -  \sum_{j=0}^{N-1} \Lambda_j P_j \frac{\partial}{\partial P_j} ,
\label{eq:genLP}
\end{align}
where $\Gamma_{jl}$ is defined by (\ref{def:Gammalj}) and $\Lambda_j $ is defined by (\ref{def:Lambdaj}).
\end{corollary}

The coefficients $\Gamma_{jl}$ describe the effective mode coupling between guided modes due to random scattering.
The coefficients $\Lambda_j$ are effective mode-dependent dissipation coefficients and they come
from the coupling between guided and radiative modes due to random scattering.

From the form of the generator ${\cal L}_{\itbf P}$, we can establish that the $n$th-order moments 
of the mode powers satisfy closed equations. We will apply this to compute the first moments of $\itbf P$, 
as well as its second moments later in Section~\ref{sec:fluctuationanalysis}.

Using (\ref{eq:genLP}) we find that the mean mode powers
\begin{equation}
\label{def:meanmodepowers}
{Q}_j(z)=
\EE [P_j(z)]  %= \lim_{\eps \to 0} \EE \big[ |\mathfrak{a}_j^\eps(z)|^2\big]
\end{equation}
satisfy the closed system of equations
\begin{equation}
\label{eq:momQ}
\partial_z {Q}_j = -\Lambda_j {Q}_j +\sum_{l=0}^{N-1} \Gamma_{jl} \big({Q}_l-{Q}_j\big)  ,
\end{equation}
starting from ${Q}_j(0)=|{a}_{j,{\rm s}}|^2$.
The form of these coupled-mode equations is well-known \cite{dozier} although the mode-dependent attenuation terms $\Lambda_j$ 
are usually introduced heuristically.
The solution explicitly writes:
\begin{align}
\label{eq:mom1b}
{\itbf Q}(z) = \exp({\bf A} z) {\itbf Q}(0)  ,
\end{align}
with the matrix ${\bf A}$ defined by ($\delta_{jl}$ is the Kronecker symbol): % and $\Gamma_{jj} = -\sum_{l'\neq j}\Gamma_{jl'} $):
\begin{align}
{\bf A} := (\Gamma_{jl} - \Lambda_j \delta_{jl})_{j,l=0}^{N-1} .
\end{align}
We can also remark that (\ref{eq:momQ}) with $\Lambda_j=0$ can be interpreted as the Kolmogorov equation associated to a random walk on the finite space $\{0,\ldots,N-1\}$. If we denote by $(J_z)_{z\geq 0}$ the Markov process on the state space $\{0,\ldots,N-1\}$ with infinitesimal generator 
$\boldsymbol{\Gamma}$, then {a} Feynman-Kac formula gives the following probabilistic representation of the mean mode powers $Q_j(z)$:
$$
Q_j(z) = \EE \Big[  |{a}_{J_z,{\rm s}}|^2 \exp\Big( -  \int_0^z \Lambda_{J_{z'}} dz' \Big) \Big| J_0=j\Big]  .
$$
This representation makes it possible to anticipate the result derived below in the continuum approximation (when the number of modes becomes large), namely that the $Q_j$'s can be approximated by the solution of a diffusion equation, because the normalized random walk $(J_z/N)_{z\geq 0}$ can be approximated in distribution by a diffusion process on $[0,1]$.

\section{Long-range behavior of the mean mode powers}
\label{sec:effmarkov3}
From now on we assume that the symmetric matrix $\boldsymbol{\Gamma}$ 
defined by $\Gamma_{jl}$ given by (\ref{def:Gammalj}) for $j\neq l$,  $\Gamma_{jj} = -\sum_{l'\neq j}\Gamma_{jl'} $, is irreducible. 
We consider the matrix  ${\bf A}=\boldsymbol{\Gamma} -\boldsymbol{\Phi} $,
with  $\Phi_{jl} = \Lambda_j \delta_{jl}$.
By Perron-Frobenius theorem, the first eigenvalue of ${\bf A}$ is simple and  nonpositive (we denote it by $-\lambda$) 
and the components of the corresponding  unit eigenvector
${\itbf V} $ have all the same sign  (so we can assume that they are positive).
By (\ref{eq:mom1b}) we get the following result.

\begin{proposition}
The mean mode powers (\ref{def:meanmodepowers}) satisfy
\begin{equation}
{Q}_j(z) \stackrel{z \to +\infty}{\simeq}
c_V V_j  \exp\big( -\lambda  z \big) \big(1+o(1)\big)  ,
\end{equation}
where $(-\lambda,{\itbf V})$ is the first eigenvalue/eigenvector of ${\bf A}$ and
\begin{equation}
c_V = \sum_{l=0}^{N-1} V_l  |{a}_{l,{\rm s}}|^2  .
\end{equation}
\end{proposition}

In the following we discuss special cases where explicit expressions can be obtained.

{\bf No effective dissipation.}
If there is no effective dissipation $\boldsymbol{\Phi} ={\bf 0}$, 
then the first eigenvalue/eigenvector $(-\lambda^{(0)},{\itbf V}^{(0)})$ of the matrix
$\boldsymbol{\Gamma} $ is 
\begin{equation}
\label{eq:defV0}
\lambda^{(0)}=0,\quad
{\itbf V}^{(0)} = \big(1/\sqrt{N} \big)_{j=0}^{N-1}, 
\end{equation}
which gives the standard equipartition result \cite{creamer,FGPSbook,GP07}:
\begin{equation}
\label{eq:equip1}
{Q}_j(z) \stackrel{z \to +\infty}{\longrightarrow}
\frac{1}{N}  \sum_{l=0}^{N-1}  |{a}_{l,{\rm s}}|^2  ,\quad \forall j= 0,\ldots,N-1 .
\end{equation}
The total input power $ \sum_{l=0}^{N-1}  |{a}_{l,{\rm s}}|^2  $ becomes equipartitioned amongst all guided  modes.

{\bf Weak effective dissipation.}
We next consider the case when the effective dissipation is weak, 
that is to say,  the matrix $\boldsymbol{\Phi} $ is much smaller
  than the matrix $\boldsymbol{\Gamma} $, with a typical ratio of the order of $\theta$.
We then assume that
$\Lambda_j = \theta \Lambda_j^{(1)}  ,$
with $\theta \ll 1$.
Then we can write $\boldsymbol{\Phi}  =\theta \boldsymbol{\Phi}^{(1)}$ 
with $\Phi_{jl}^{(1)} =
\Lambda_j^{(1)} \delta_{jl}$
and $\boldsymbol{\Gamma}  = \boldsymbol{\Gamma}^{(0)}$
and the first eigenvalue/eigenvector  $(-\lambda,{\itbf V} )$  of the matrix
$\boldsymbol{\Gamma} -\boldsymbol{\Phi} $ can be expanded
as
$$
 \lambda=\lambda^{(0)}+ \theta \lambda^{(1)} +\theta^2 \lambda^{(2)} +O(\theta^3),\quad
 {\itbf V}  = {\itbf V}^{(0)} +\theta {\itbf V}^{(1)} +O(\theta^2)
$$
with $(\lambda^{(0)}, {\itbf V}^{(0)})$ given by (\ref{eq:defV0}),
\begin{align}
\label{eq:lamdab1sc}
\lambda^{(1)} =& {{\itbf V}^{(0)}}^T \boldsymbol{\Phi}^{(1)} {\itbf V}^{(0)}
=
\frac{1}{N}\sum_{j=0}^{N-1} \Lambda_j^{(1)},\\
\lambda^{(2)} =& {{\itbf V}^{(0)}}^T 
%\boldsymbol{\Gamma}^{(0)} 
\boldsymbol{\Phi}^{(1)}
{\itbf V}^{(1)} ,
\end{align}
and $ {\itbf V}^{(1)}$ is solution of $ \boldsymbol{\Gamma}^{(0)}  {\itbf V}^{(1)} = (\boldsymbol{\Phi}^{(1)} - \lambda^{(1)} )
{\itbf V}^{(0)}$ and is orthogonal to $ {\itbf V}^{(0)}$.
If, for instance, $\Gamma_{jl}\equiv \Gamma >0$ for all $j\neq l$, then
\begin{equation}
\label{eq:V1sc}
{\itbf V}^{(1)} = -\frac{1}{\Gamma N^{3/2}} \big(\Lambda_j^{(1)} -\lambda^{(1)} \big)_{j=0}^{N-1}
\end{equation}
%$$
%{\itbf V}^{(1)} = -\frac{1}{\Gamma N^{3/2}} \big(\Lambda_j^{(1)} \big)_{j=0}^{N-1}
%$$
and
\begin{equation}
\lambda^{(2)} = 
- \frac{1}{\Gamma N^2} \sum_{j=0}^{N-1} (\Lambda_j^{(1)} -\lambda^{(1)})^2 .
\label{eq:lamdab2sc}
\end{equation}
Eqs.~(\ref{eq:defV0})-(\ref{eq:lamdab1sc}) show that, when coupling is stronger than dissipation, 
then the effective damping of the mean mode powers is approximately the arithmetic 
average of the effective mode-dependent damping coefficients. 
Eqs.~(\ref{eq:defV0})-(\ref{eq:V1sc}) show that the distribution of the mean mode powers is approximately equipartitioned,
with reduced allocations for the modes with the strongest damping coefficients.

{\bf Weak coupling.}
We next consider the case when the coupling is weak, that is to say,
 the matrix $\boldsymbol{\Gamma} $  is much smaller
  than the matrix $\boldsymbol{\Phi} $, with a typical ratio of the order of $\theta$.
We then assume that
$
\Gamma_{jl} = \theta \Gamma_{jl}^{(1)}  ,
$
with $\theta \ll 1$.
We also assume that the dissipation coefficients have a unique minimum 
\begin{equation}
\label{eq:uniquemin}
 \Lambda_{j_\star}  = \min_{j=0,\ldots,N-1} (\Lambda_j ), \quad \quad \Lambda_j > \Lambda_{j_\star} \quad \forall j \neq j_\star.
\end{equation}
Then we can write $\boldsymbol{\Phi}  = \boldsymbol{\Phi}^{(0)}$ and $\boldsymbol{\Gamma}  = \theta \boldsymbol{\Gamma}^{(1)}$
and the first eigenvalue/eigenvector  $(-\lambda,{\itbf V} )$  of the matrix
$\boldsymbol{\Gamma} -\boldsymbol{\Phi} $ can be expanded
as
$$
 \lambda=   \lambda^{(0)} +\theta \lambda^{(1)} +\theta^2 \lambda^{(2)} +O(\theta^3),
\quad
 {\itbf V}  = {\itbf V}^{(0)} +\theta {\itbf V}^{(1)} +O(\theta^2)
  ,
  $$
with
\begin{eqnarray}
\label{eq:lamdab1wc}
\lambda^{(0)} &=&  \Lambda_{j_\star} ,\\
V^{(0)}_j &=& \delta_{jj_\star},\\
\lambda^{(1)} &=&
- {{\itbf V}^{(0)}}^T \boldsymbol{\Gamma}^{(1)} {\itbf V}^{(0)} =
-\Gamma_{j_\star j_\star}^{(1)}= \sum_{j \neq j_\star}  \Gamma_{jj_\star}^{(1)} ,\\
V^{(1)}_j &=& \frac{\Gamma_{jj_\star}^{(1)}}{\Lambda_j-\Lambda_{j_\star}} \quad \forall j \neq j_\star,  \quad  V^{(1)}_{j_\star}=0 ,\\
\lambda^{(2)} &=& 
-{{\itbf V}^{(0)}}^T \boldsymbol{\Gamma}^{(1)} {\itbf V}^{(1)} 
=
- \sum_{j \neq j_\star} \frac{(\Gamma_{jj_\star}^{(1)})^2}{\Lambda_j-\Lambda_{j_\star}} .
\end{eqnarray}
Eq.~(\ref{eq:lamdab1wc}) shows that, when coupling is weaker than dissipation, then
the effective damping of the mean mode powers is approximately the minimum of the effective mode-dependent damping coefficients.
The distribution of the mean mode powers is, moreover, concentrated on the mode corresponding to the minimal damping coefficient.

{\bf Continuum approximation.}
Here we want to address situations where the coupling between guided modes is via nearest neighbors
and the number of modes is large.

The number of modes becomes large when $(n^2-1)k^2 d^2 \gg 1$ (see (\ref{eq:numberofmodes})). In other words, the number of modes is large when the frequency is large because it is proportional to the ratio of the waveguide diameter over the wavelength.

For type II perturbations, coupling via nearest neighbors happens when the fluctuations of the boundaries are smooth so that the 
Fourier transform of ${\cal R}_{\rm II}$ decays fast and the correlation radius is larger than the core diameter (which is much larger than the wavelength). 
Under such circumstances,
we have $\beta_j-\beta_{j+l}\simeq \frac{\pi \sqrt{n^2-1}}{n d}\frac{j}{N} l$ for any $l\geq 1$ (see Appendix \ref{app:dec}), 
$|\hat{\cal R}_{\rm II}(\beta_j-\beta_{j+1})| \gg  |\hat{\cal R}_{\rm II}(\beta_j-\beta_{j+l})|$ for any $l\geq 2$, and
we can approximate the matrix $\boldsymbol{\Gamma}$ for all $l\neq j$ by:
\begin{equation}
\label{eq:approxGammajl}
\Gamma_{jl} =
\left\{
\begin{array}{ll}
\gamma_{\min(l,j)}
& \mbox{ if } |j-l|=1,\\
0 & \mbox{ if } |j-l|\geq 2,
\end{array}
\right.
\end{equation}
with
\begin{equation}
\label{def:gammajII}
\gamma_j = \frac{k^2 (n^2-1)^2 d^2}{2\beta_j \beta_{j+1}}  [\phi_j\phi_{j+1}](\frac{d}{2})^2  \hat{\cal R}_{\rm II}(\beta_j-\beta_{j+1})  .
\end{equation}

For type I perturbations, coupling via nearest neighbors  happens under similar conditions. Indeed, let us assume that ${\cal R}_{\rm I}(x,x',z)$
can be factorized as
$$
{\cal R}_{\rm I}(x,x',z) = {\cal R}_{I,{\rm c}}(x,x') {\cal R}_{I,{\rm l}} (z) ,
$$
then for all $l\neq j$:
$$
\Gamma_{jl} = \frac{k^2}{4\beta_j\beta_l} \hat{\cal R}_{I,{\rm l}} (\beta_j-\beta_l) \iint_{[-d/2,d/2]^2} \! {\cal R}_{I,{\rm c}}(x,x') \phi_j \phi_l(x) \phi_j \phi_l(x') dx dx' ,
$$
where $ \hat{\cal R}_{I,{\rm l}}$ is the Fourier transform of ${\cal R}_{I,{\rm l}}$.
Again, if the fluctuations of the index of refraction are smooth so that 
$\hat{\cal R}_{I,{\rm l}}$ decays fast and the longitudinal correlation radius is larger than the core diameter,
then we can approximate $\Gamma_{jl}$ by (\ref{eq:approxGammajl}) with
\begin{equation}
\label{def:gammajI}
\gamma_j = 
 \frac{k^2}{4\beta_j\beta_{j+1}} \hat{\cal R}_{I,{\rm l}} (\beta_j-\beta_{j+1}) \iint_{[-d/2,d/2]^2} \! {\cal R}_{I,{\rm c}}(x,x') \phi_j \phi_{j+1}(x) \phi_j \phi_{j+1}(x') dx dx' .
\end{equation}
Similarly, we find that $\Lambda_j$ can be approximated by
$$
\Lambda_j = \left\{\
\begin{array}{ll}
 \Lambda_{N-1} & \mbox{ if } j=N-1,\\
 0 & \mbox{ otherwise}.
 \end{array}
 \right.
$$
Other circumstances can lead to the same conclusions. For instance the band-limiting idealization hypothesis in \cite{gomez} gives the same result,
and it is based on the behavior of the transverse covariance function ${\cal R}_{I,{\rm c}}$.

The coupled mode power equations then read
\begin{align}
\label{eq:QNNa}
\partial_z Q_{N-1}^{(N)} &= - \Lambda_{N-1}^{(N)} Q_{N-1}^{(N)}+\gamma_{N-2}^{(N)} (Q_{N-2}^{(N)}-Q_{N-1}^{(N)}) ,\\
\partial_z Q_j^{(N)} &= \gamma_{j-1}^{(N)}   (Q_{j-1}^{(N)}-Q_j^{(N)})   +\gamma_j^{(N)} (Q_{j+1}^{(N)}-Q_j^{(N)})  \mbox{ for } 1\leq j \leq N-2,\\
\partial_z Q_0^{(N)} &= \gamma_0^{(N)}(Q_1^{(N)}-Q_0^{(N)}).
\label{eq:QNNb}
\end{align}
The superscript $(N)$ is added to remember that all coefficients depend on $N$.
We have nearest-neighbor
coupling: The $j$th mode can exchange power with the $l$th mode only if they are neighbors, that is, if they satisfy $|j-l|\leq 1$.

\begin{figure}
\begin{center}
\begin{tikzpicture}[scale=7]
\draw [->] (-0.1,0) --  (0.9,0) ;
\draw (0.9,0) node[right] {$j$};
\draw [very thick] (0,0) --  (0.8,0) ;

\draw (0,0) node{$\bullet$};
\draw (0,0) node[below] {$0$};
\draw[->,>=latex] (0,0)  to[bend  left] (0.1 , 0);
\draw (0.08,0) node[above] {$\gamma_{0}$};

\draw (0.4,0) node{$\bullet$};
\draw (0.4,0) node[below] {$j$};
\draw[->,>=latex] (0.4,0)  to[bend  left] (0.5 , 0);
\draw (0.48,0) node[above] {$\gamma_{j}$};
\draw[->,>=latex] (0.4,0)  to[bend  right] (0.3 , 0);
\draw (0.32,0) node[above] {$\gamma_{j-1}$};

\draw (0.8,0) node{$\bullet$};
\draw (0.8,0) node[below] {$N-1$};
\draw[->,>=latex] (0.8,0)  to[bend  right] (0.7 , 0);
\draw (0.68,0) node[above] {$\gamma_{N-2}$};
\end{tikzpicture}

\end{center}
\caption{The transition rates of the jump Markov process $(J_z)_{z\geq 0}$ on the state space $ \{0 \leq j  \leq N-1\}$. 
The absorption is concentrated on the point $j=N-1$. }
\label{fig:markov1}
\end{figure}
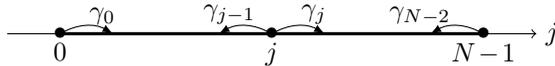

We note that the terms involving $\gamma_j^{(N)}$ in (\ref{eq:QNNa}-\ref{eq:QNNb})
define the infinitesimal generator of a random walk on the finite set $\{0 \leq j \leq N-1\}$ (see Figure \ref{fig:markov1}).
As shown in \cite{gomez}, following the ideas developed in \cite[Chapter 11]{sv},
if $\gamma_j^{(N)}$ converges in the sense that
$\gamma_{\lfloor xN\rfloor}^{(N)} \to \gamma_\infty( x)$ for any $x \in (0,1)$, where $\gamma_\infty$ is smooth and positive,
then, for any function $\varphi$ in $L^2(0,1)$,
the function 
$$
Q^{(N)}_\varphi (z,u) = \sum_{j=0}^{N-1} Q^{(N)}_{j} (z) \varphi(\lfloor j/N\rfloor) ,
$$
where $Q^{(N)}_j$ is the solution of (\ref{eq:QNNa}-\ref{eq:QNNb}) starting from $Q^{(N)}_j(z=0)=\delta_{j , \lfloor u N \rfloor}$,
satisfies
$$
\lim_{N \to \infty} Q^{(N)}_\varphi (z,u)  = Q_\varphi(z,u) ,
$$
where $Q_\varphi$ is the solution of the partial differential equation:
\begin{equation}
\partial_z Q_\varphi = {\cal H}_1 Q_\varphi, \quad \quad 
 {\cal H}_1 = \partial_u \big( \gamma_\infty(u)\partial_u \cdot \big)  ,
\end{equation}
with the mixed Neumann-Dirichlet boundary conditions
\begin{equation}
\partial_u Q_\varphi(z,0)=0, \quad
Q_\varphi(z,1)=0, \quad
Q_\varphi(0,u)=\varphi(u).
\end{equation}
For type I perturbations we have from  (\ref{def:gammajI}):
\begin{align*}
\gamma_\infty(u) =& 
\frac{1}{4 (n^2-1)d^2 (\frac{n^2}{n^2-1}-u^2) }  
 \hat{\cal R}_{I,{\rm l}}\big( \frac{\pi \sqrt{n^2-1}}{n d} u\big)   \\
& \times
  \iint_{[-d/2,d/2]^2} \! {\cal R}_{I,{\rm c}}(x,x')  \sin( \pi x /d)  \sin(\pi x'/d) dx dx'.
\end{align*}
For type II perturbations, we have from (\ref{def:gammajII}):
$$
\gamma_\infty(u) = 
\frac{2 (n^2-1) u^4}{ (\frac{n^2}{n^2-1}-u^2) }  
 \hat{\cal R}_{\rm II}\big( \frac{\pi \sqrt{n^2-1}}{n d} u\big)  .
$$
As a consequence of this result we get the following result.
\begin{proposition}
\label{prop:lambdacont}
In the continuum approximation, the first eigenvalue $ - \lambda^{(N)}$ of the matrix ${\bf A}^{(N)}$ converges to $-\lambda$ as $N\to \infty$,
with
\begin{equation}
\lambda = \inf_{\varphi \in {\cal D}_1}
\int_0^1 \gamma_\infty(u) \varphi'(u)^2 du
\end{equation}
and
\begin{equation}
{\cal D}_1 = \Big\{ \varphi \in {\cal C}^\infty([0,1]) , \, \int_0^1 \varphi(u)^2 du=1, \, \varphi'(0)=0 , \,  \varphi(1)=0 \Big\} .
\end{equation}
\end{proposition}
Moreover, $-\lambda$ is a simple eigenvalue of the operator ${\cal H}_1$, the corresponding eigenvector $\varphi$ is smooth and unique 
(up to a multiplication by $-1$) and it can be chosen so as to satisfy
$\varphi(u)>0$ for $u \in [0,1)$ (the proof is similar as the one proposed in \cite{gomez} for the Pekeris waveguide
and it is based on a modified version of Krein-Rutman theorem).
The eigenvector $\varphi$ gives the asymptotic mode distribution for large propagation distance:
\begin{equation}
Q_j^{(N)}(z) \stackrel{z \to +\infty}{\simeq} c_V \varphi(j /N) \exp(-\lambda z),
\end{equation}
with $c_V=\sum_{l=0}^{N-1} |a_{l,{\rm s}}|^2 \varphi(l/N)/N$. We, therefore, observe an exponential decay of the mean power transported by the guided modes and a form of equipartition of the mean mode powers,
but not with the uniform distribution, but with the distribution proportional to the eigenvector $\varphi$.

If $\gamma_\infty$ is constant, then $\lambda= \pi^2 \gamma_\infty/4$ and the eigenvector is 
$\varphi(u)=\sqrt{2} \cos(\pi u/2)$. 
This happens in particular for type I perturbations when $0<n-1\ll 1$, so that, for all $u \in (0,1)$,
\begin{align*}
\gamma_\infty(u) \simeq &
\frac{1}{4 d^2  }  
 \hat{\cal R}_{I,{\rm l}}(0)    \iint_{[-d/2,d/2]^2} \! {\cal R}_{I,{\rm c}}(x,x')   \sin( \pi x /d)  \sin(\pi x'/d) dx dx'.
\end{align*}

\section{Fluctuation analysis}
\label{sec:fluctuationanalysis}
By (\ref{eq:genLP}) we find that the second-order moments of the mode powers 
\begin{equation}
R_{jl}  (z)= \EE \big[P_j(z)P_l(z)\big] ,\quad j,l=0,\ldots,N-1,
\end{equation}
satisfy the closed equations
\begin{align}
 \partial_z R_{jj} =& - 2 \Lambda_j R_{jj}  +\sum_{n \neq j}  \Gamma_{jn} (4R_{jn} -2R_{jj} )  ,\\
\nonumber
 \partial_z R_{jl}  =& - (2\Gamma_{jl}+ \Lambda_j +\Lambda_l) R_{jl}  +
\sum_{n \neq l}  \Gamma_{ln} (R_{jn} -R_{jl} ) \\
& + \sum_{n \neq j}  \Gamma_{jn} (R_{nl} -R_{jl} )  , \quad j\neq l.
\end{align}
This system has the same form as the one found in the literature dedicated to coupled mode theory \cite{dozier,creamer}.
The initial conditions are $ R_{jl} (0)=|{a}_{j,{\rm s}}|^2|{a}_{l,{\rm s}}|^2 $.
Let us introduce ${\itbf S} = (S_{jl})_{0\leq j\leq l\leq N-1}$ defined by
\begin{equation}
S_{jl} = 
\left\{
\begin{array}{ll}
R_{jl}+R_{lj} =2R_{jl} &\mbox{ if } j  <  l,\\
R_{jj} & \mbox{ if } j=l  .
\end{array}
\right.
\end{equation}
The $S_{jl}$'s satisfy the system
\begin{align}
\label{eq:Sjl1}
\partial_z S_{jl} =&
- ( \boldsymbol{\Psi} {\itbf S})_{jl} +(\boldsymbol{\Theta}  {\itbf S})_{jl}, \\
( \boldsymbol{\Psi} {\itbf S})_{jl}=&   (\Lambda_j+\Lambda_l)S_{jl} ,\\
\nonumber
(\boldsymbol{\Theta}  {\itbf S})_{jl}=& 2\Gamma_{jl} {\bf 1}_{j \neq l} (S_{jj}+S_{ll} -2S_{jl}) 
\\
&+\sum_{n \not\in \{j,l\}} \big[ \Gamma_{ln} (S_{jn}-S_{jl})+\Gamma_{jn}(S_{nl}-S_{jl})\big] ,
\end{align}
with the convention that whenever $S_{jl}$ occurs with $j>l$, it is replaced by $S_{lj}$.
The operator $\boldsymbol{\Theta} $ is the infinitesimal generator of a random Markov process $(J_z,L_z)_{z\geq 0}$
that is a random walk on the discrete triangle $\{ (j,l)\in \NN^2, \, 0 \leq j \leq l \leq N-1\}$.
Using a Feynmac-Kac formula we get the following probabilistic representation of $S_{jl}$:
$$
S_{jl}(z) = \EE\Big[ |a_{J_z,{\rm s}}|^2  |a_{L_z,{\rm s}}|^2 \exp\Big( - \int_0^z \Lambda_{J_{z'}} + \Lambda_{L_{z'}} dz' \Big) \Big|J_0=j, L_0=l \Big]  .
$$
We can anticipate that, in the continuum limit, the Markov process $(J_z/N,L_z/N)_{z\geq 0}$ behaves as a diffusion process on 
the triangle $\{(u,v)\in \RR^2, \, 0\leq u \leq v\leq 1\}$, and, therefore, $S_{jl}$ satisfies a diffusion equation on this triangle.

{\bf Long-range behavior of the second-order moments of the mode powers.}
Eq.~(\ref{eq:Sjl1}) has the form $\partial_z {\itbf S}  = (\boldsymbol{\Theta} 
- \boldsymbol{\Psi} ) {\itbf S}$.
The linear operator $ \boldsymbol{\Psi} $ is diagonal and the linear operator $\boldsymbol{\Theta} $
 is self-adjoint:  for any ${\itbf T}$ and $\widetilde{\itbf T}$, we have
\begin{align*}
\sum_{j\leq l}  (\boldsymbol{\Theta}  {\itbf T})_{jl}  \widetilde{T}_{jl}
&= -\sum_{j\leq l} \Theta_{jl,jl} T_{jl} \widetilde{T}_{jl}
+
\sum_{j <l , n \not\in \{j,l\}} \big( \Gamma_{ln} T_{jn} \widetilde{T}_{jl} +\Gamma_{jn}T_{nl}\widetilde{T}_{jl}\big) \\
&\quad +\sum_{j \neq n}\big(\Gamma_{jn} T_{jn} \widetilde{T}_{jj} +\Gamma_{jn} T_{nj} \widetilde{T}_{jj}\big) 
+2 \sum_{j < l}\big( \Gamma_{jl} T_{jj} \widetilde{T}_{jl} + \Gamma_{jl} T_{jj} \widetilde{T}_{jl}\big)\\
&=\sum_{j\leq l}
T_{jl} (\boldsymbol{\Theta}  \widetilde{\itbf T})_{jl} 
,
\end{align*}
because $2\sum_{j<l} = \sum_{j \neq l}$.
As a consequence, $\boldsymbol{\Theta} 
- \boldsymbol{\Psi} $ can be diagonalized and as a consequence of Perron-Frobenius theorem
we get the following result.
\begin{proposition}
The second-order moments of the mode powers satisfy
\begin{equation}
{\itbf S}(z) \stackrel{z \to +\infty}{\simeq}
c_W {\itbf W} \exp\big( -\mu z \big) \big(1+o(1)\big)  ,
\end{equation}
where 
$(-\mu,{\itbf W})$ is the first eigenvalue/eigenvector of $\boldsymbol{\Theta}-\boldsymbol{\Psi}$
and $c_W$ is the projection of the initial conditions on the first eigenvector $ {\itbf W}$ 
\begin{equation}
c_W 
=
 \sum_{j,l=0}^{N-1} W_{jl}  |{a}_{j,{\rm s}}|^2|{a}_{l,{\rm s}}|^2  ,
\end{equation}
with the convention that whenever $W_{jl}$ occurs with $j>l$, it is replaced by $W_{lj}$.
\end{proposition}

We next address special cases.

{\bf No effective dissipation.}
If there is no effective dissipation, then the first eigenvalue/eigenvector  $(-\mu^{(0)},{\itbf W}^{(0)})$ of the matrix
$\boldsymbol{\Theta}$ is 
$
{\itbf W}^{(0)} = \big( c_N \big)_{0\leq j \leq l \leq N-1}$, 
$\mu^{(0)}=0$,
with $c_N=\sqrt{2}/\sqrt{N(N+1)}$.
We have
$
{\itbf S}(z) \stackrel{z \to +\infty}{\to}
c_W {\itbf W^{(0)}}  .
$
As $\sum_{j\leq l} S_{jl}(z) =\sum_{j,l} R_{jl}(z) = (\sum_{j=0}^{N-1} |{a}_{j,{\rm s}}|^2)^2$, we deduce 
$$
{S}_{jl}(z) \stackrel{z \to +\infty}{\to} \Big(\sum_{l'=0}^{N-1} |{a}_{l',{\rm s}}|^2\Big)^2\frac{2}{N(N+1)} ,
$$
and
$$
{R}_{jl}(z) \stackrel{z \to +\infty}{\to} \Big(\sum_{l'=0}^{N-1} |{a}_{l',{\rm s}}|^2\Big)^2\frac{1+\delta_{jl}}{N(N+1)}.
$$
{By taking into account (\ref{eq:equip1}), 
this means that, when $N \gg 1$, the mode powers $P_j$ become uncorrelated.
Furthermore, this regime was analyzed in detail in \cite[Chapter 20]{FGPSbook},
where it is shown that the marginal distributions of the mode powers $P_j$ 
have the same moments as exponential distributions.} In other words, the mode powers behave as the square moduli 
of independent and identically distributed complex Gaussian variables.

{\bf Weak effective dissipation.}
We next consider the case when the effective dissipation is weak, 
say $\Lambda_j = \theta \Lambda_j^{(1)}$ with $\theta \ll 1$.
Then we can write $\boldsymbol{\Psi} =\theta \boldsymbol{\Psi}^{(1)}$ and $\boldsymbol{\Theta} = \boldsymbol{\Theta}^{(0)}$
and the first eigenvalue/eigenvector $(-\mu,{\itbf W})$  of the matrix
$\boldsymbol{\Theta}-\boldsymbol{\Psi}$ can be expanded
as
$$
 {\itbf W} = {\itbf W}^{(0)} +\theta {\itbf W}^{(1)} +O(\theta^2)
  ,\quad  \mu= \theta \mu^{(1)} +\theta^2 \mu^{(2)} +O(\theta^3),
$$
with
\begin{align}
\nonumber
\mu^{(1)} =& {{\itbf W}^{(0)}}^T \boldsymbol{\Psi}^{(1)} {\itbf W}^{(0)}
 =\frac{2}{N} \sum_{j=0}^{N-1}  \Lambda_j^{(1)} = 2\lambda^{(1)},\\
\mu^{(2)} =& {{\itbf W}^{(1)}}^T \boldsymbol{\Theta}^{(0)} {\itbf W}^{(1)} ,
\end{align}
and $ {\itbf W}^{(1)}$ is solution of $ \boldsymbol{\Theta}^{(0)}  {\itbf W}^{(1)} =( \boldsymbol{\Psi}^{(1)} -\mu^{(1)})
{\itbf W}^{(0)}$ and is orthogonal to $ {\itbf W}^{(0)}$.
If, for instance, $\Gamma_{jl}\equiv \Gamma >0$ for all $j\neq l$, then
$$
W_{jl}^{(1)}= - \frac{c_N}{\Gamma N} \big( \Lambda_j^{(1)}+ \Lambda_l^{(1)} -2\lambda^{(1)}\big) , \quad j\leq l,
$$
and
\begin{align*}
\mu^{(2)}  &=
{{\itbf W}^{(1)}}^T \boldsymbol{\Psi}^{(1)} {\itbf W}^{(0)} \\
%\sum_{j\leq l} W_{jl}^{(1)} (  \boldsymbol{\Psi}^{(1)}  {\itbf W}^{(0)} )_{jl}\\
%\sum_{j\leq l} W_{jl}^{(1)} (  \boldsymbol{\Theta}^{(0)}  {\itbf W}^{(1)} )_{jl}\\
&= - \frac{2(N+2)}{N^2(N+1) \Gamma}  \sum_{j=0}^{N-1} (\Lambda_j^{(1)}-\lambda^{(1)})^2  .
\end{align*}
Note that
\begin{align}
\nonumber
\mu - 2\lambda  =&  
\theta^2\big(\mu^{(2)}  -2 \lambda^{(2)}\big) +O(\theta^3)
\\
=&
-\frac{2 \theta^2}{N^2(N+1) \Gamma }   \sum_{j=0}^{N-1} (\Lambda_j^{(1)}-\lambda^{(1)})^2 
+O(\theta^3) 
\label{eq:diffvp}
\end{align}
is negative-valued as soon as there exist $j,j'$ such that $\Lambda_j^{(1)}\neq \Lambda_{j'}^{(1)}$.

{\bf Weak coupling.}
We next consider the case when the coupling is weak, 
say
$
\Gamma_{jl} = \theta \Gamma_{jl}^{(1)} 
$,
with $\theta \ll 1$.
We again assume that the dissipation coefficients have a unique minimum 
(\ref{eq:uniquemin}).
  Then we can write $\boldsymbol{\Psi} = \boldsymbol{\Psi}^{(0)}$ and $\boldsymbol{\Theta} = \theta \boldsymbol{\Theta}^{(1)}$
and the first eigenvalue/eigenvector  $(-\mu,{\itbf W})$  of the matrix
$\boldsymbol{\Theta}-\boldsymbol{\Psi}$ can then be expanded
as
$$
 {\itbf W} = {\itbf W}^{(0)} +\theta {\itbf W}^{(1)} +O(\theta^2)
  ,\quad  \mu=   \mu^{(0)} +\theta \mu^{(1)} +\theta^2 \mu^{(2)} +O(\theta^3),
$$
with
\begin{eqnarray}
\mu^{(0)} &=& 2 \Lambda_{j_\star} ,\\
W^{(0)}_{jl} &=& \delta_{jj_\star} \delta_{lj_\star},\\
\mu^{(1)} &=& - {{\itbf W}^{(0)}}^T \boldsymbol{\Theta}^{(1)} {\itbf W}^{(0)} =
 2 \sum_{j \neq j_\star} \Gamma_{jj_\star}^{(1)} ,
 \end{eqnarray}
 $ {\itbf W}^{(1)}$ is solution of $ (\mu^{(0)}-\boldsymbol{\Phi}^{(0)} ) {\itbf W}^{(1)} =(- \boldsymbol{\Theta}^{(1)} -\mu^{(1)})
{\itbf W}^{(0)}$ and is orthogonal to $ {\itbf W}^{(0)}$,
\begin{eqnarray}
W^{(1)}_{jl} &=& 
\left\{
\begin{array}{ll}
\frac{ 2\Gamma_{j_\star l} }{\Lambda_l - \Lambda_{j_\star}}   & \mbox{ if } j=j_\star, \, l >j_\star,\\
\frac{ 2\Gamma_{jj_\star} }{\Lambda_j - \Lambda_{j_\star}}  & \mbox{ if } j<j_\star, \, l =j_\star,\\
0 &\mbox{ otherwise},
\end{array}
\right.
\\
\nonumber
\mu^{(2)} &=& 
- {{\itbf W}^{(0)}}^T \boldsymbol{\Theta}^{(1)} {\itbf W}^{(1)} 
= -4 \sum_{j\neq j_\star} \frac{(\Gamma_{jj_\star}^{(1)})^2}{\Lambda_j-\Lambda_{j_\star}} .
\end{eqnarray}
Note that
\begin{eqnarray}
\mu - 2\lambda =  -2 \theta^2 \sum_{j\neq j_\star} \frac{(\Gamma_{jj_\star}^{(1)})^2}{\Lambda_j-\Lambda_{j_\star}}  +O(\theta^3) 
\label{eq:diffvpbis}
\end{eqnarray}
is negative-valued (we have assumed the irreducibility of $\boldsymbol{\Gamma}^{(1)}$, hence at least one of the $\Gamma_{jj_\star}^{(1)}$
is non-zero).

% $ {\itbf W}^{(1)}$ is solution of $ (\mu^{(0)}-\boldsymbol{\Phi}^{(0)} ) {\itbf W}^{(1)} =(- \boldsymbol{\Theta}^{(1)} -\mu^{(1)})
%{\itbf W}^{(0)}$ and is orthogonal to $ {\itbf W}^{(0)}$,
%\begin{eqnarray}
%W^{(1)}_{jl} &=& \frac{ \Gamma_{j_\star l} }{\Lambda_l - \Lambda_{j_\star}} \delta_{jj_\star} (1-\delta_{l j_\star})  ,\\
%\nonumber
%\mu^{(2)} &=& 
%- {{\itbf W}^{(0)}}^T \boldsymbol{\Theta}^{(1)} {\itbf W}^{(1)} 
%= -2 \sum_{j\neq j_\star} \frac{(\Gamma_{jj_\star}^{(1)})^2}{\Lambda_j-\Lambda_{j_\star}} .
%\end{eqnarray}
%Note that
%\begin{eqnarray}
%\mu - 2\lambda =  O(\theta^3) .
%\label{eq:diffvpbis}
%\end{eqnarray}
%

{\bf Continuum approximation.}
Here we address the situations where the coupling between guided modes is via nearest neighbors
and the number of modes is large. When $\Gamma_{jl}$ is of the form (\ref{eq:approxGammajl}),
the system (\ref{eq:Sjl1}) for $S_{jl}$ reads
\begin{align}
\nonumber
\partial_z S_{jl} =& \delta_{jl} \big[ 2\gamma_j (S_{jj+1}-S_{jj}) {\bf 1}_{j\leq N-2} 
+2\gamma_{j-1} (S_{j-1j}-S_{jj} ){\bf 1}_{j \geq 1}\big] \\
\nonumber
& + \delta_{l j+1} \big[ 2 \gamma_j (S_{jj}+S_{j+1j+1}-2S_{jj+1}) 
+\gamma_{j-1} (S_{j-1j+1}-S_{jj+1}) {\bf 1}_{j \geq 1} \\
\nonumber
&\quad +\gamma_{j+1} (S_{jj+2}-S_{jj+1}){\bf 1}_{j \leq N-3} \big]\\
\nonumber
&+ {\bf 1}_{j\leq l-2} \big[ \gamma_{l-1} (S_{jl-1}-S_{jl} ){\bf 1}_{l \geq 1}
+\gamma_l (S_{jl+1}-S_{jl}) {\bf 1}_{l\leq N-2}\\
\nonumber&\quad 
+\gamma_{j-1} (S_{j-1l}-S_{jl}){\bf 1}_{j \geq 1}
+\gamma_j (S_{j+1l}-S_{jl}) %{\bf 1}_{j \leq N-2}
\big] \\
&
- \Lambda_{N-1}(\delta_{jN-1}+\delta_{lN-1}) S_{jl} .
\label{eq:SNN}
\end{align}
Note that the terms involving $\gamma_j$ define the infinitesimal generator of a random walk 
$(J_z,L_z)_{z\geq 0}$ on the finite set $D_N=\{0 \leq j \leq l \leq N-1\}$ (see Figure \ref{fig:markov2}).

\begin{figure}
\begin{center}
\begin{tikzpicture}[scale=7]
\draw [->] (0,0) --  (0,0.9 );
\draw (0,0.9) node[right] {$l$};
\draw [->] (0,0) --  (0.9,0) ;
\draw (0.9,0) node[right] {$j$};
\draw (0,0) node[left] {$(0,0)$};
\draw (0,0.8) node[left] {$(0,N-1)$};
\draw (0.8,-0.01) -- (0.8,0.01);
\draw (0.8,0) node[below] {$(N-1,0)$};
\draw [very thick] (0,0) --  (0,0.8) -- (0.8,0.8) -- (0,0) ;

\draw (0.2,0.3) node{$\bullet$};
\draw (0.24,0.23) node[right] {$(j,j+1)$};
\draw[->,>=latex] (0.2, 0.3)  to[bend right] (0.2 , 0.2);
\draw (0.2,0.21) node[left] {$2\gamma_j$};
\draw[->,>=latex] (0.2, 0.3)  to[bend  left] (0.3 , 0.3);
\draw (0.28,0.3) node[above] {$2\gamma_j$};
\draw[->,>=latex] (0.2, 0.3)  to[bend  right] (0.2 , 0.4);
\draw (0.185,0.4) node[right] {$\gamma_{j+1}$};
\draw[->,>=latex] (0.2, 0.3)  to[bend  left] (0.1 , 0.3);
\draw (0.04,0.32) node[right] {$\gamma_{j-1}$};

\draw (0.6,0.6) node{$\bullet$};
\draw (0.59,0.58) node[right] {$(j,j)$};
\draw[->,>=latex] (0.6, 0.6)  to[bend  right] (0.6 , 0.7);
\draw (0.58,0.62) node[left] {$2\gamma_{j-1}$};
\draw[->,>=latex] (0.6, 0.6)  to[bend  left] (0.5 , 0.6);
\draw (0.62,0.71) node[left] {$2\gamma_j$};

\draw (0,0.5) node{$\bullet$};
\draw (0,0.5) node[left] {$(0,l)$};
\draw[->,>=latex] (0, 0.5)  to[bend  right] (0 , 0.6);
\draw (0,0.59) node[right] {$\gamma_{l}$};
\draw[->,>=latex] (0, 0.5)  to[bend  left] (0 , 0.4);
\draw (0,0.41) node[right] {$\gamma_{l-1}$};
\draw[->,>=latex] (0, 0.5)  to[bend  left] (0.1 , 0.5);
\draw (0.08,0.51) node[right] {$\gamma_0$};

\draw (0.3,0.6) node{$\bullet$};
\draw (0.3,0.55) node[left] {$(j,l)$};
\draw[->,>=latex] (0.3,0.6)  to[bend  right] (0.3 , 0.7);
\draw (0.295,0.69) node[right] {$\gamma_{l}$};
\draw[->,>=latex] (0.3,0.6)  to[bend  right] (0.3 , 0.5);
\draw (0.29,0.49) node[right] {$\gamma_{l-1}$};
\draw[->,>=latex] (0.3,0.6)  to[bend  left] (0.4 , 0.6);
\draw (0.38,0.58) node[right] {$\gamma_{j}$};
\draw[->,>=latex] (0.3,0.6)  to[bend  left] (0.2 , 0.6);
\draw (0.14,0.62) node[right] {$\gamma_{j-1}$};

\draw (0.8,0.8) node{$\bullet$};
\draw (0.8,0.8) node[right] {$(N-1,N-1)$};
\draw[->,>=latex] (0.8,0.8)  to[bend  right] (0.7 , 0.8);
\draw (0.71,0.8) node[above] {$\gamma_{N-2}$};

\draw (0.4,0.8) node{$\bullet$};
\draw (0.275,0.8) node[below] {$(j,N-1)$};
\draw[->,>=latex] (0.4,0.8)  to[bend  right] (0.3 , 0.8);
\draw (0.31,0.8) node[above] {$\gamma_{j-1}$};
\draw[->,>=latex] (0.4,0.8)  to[bend  left] (0.5 , 0.8);
\draw (0.48,0.8) node[above] {$\gamma_{j}$};
\draw[->,>=latex] (0.4,0.8)  to[bend  right] (0.4 , 0.7);
\draw (0.385,0.7) node[right] {$\gamma_{N-2}$};

\draw (0,0.8) node{$\bullet$};
\draw[->,>=latex] (0,0.8)  to[bend  left] (0.1 , 0.8);
\draw (0.1,0.8) node[above] {$\gamma_{0}$};
\draw[->,>=latex] (0,0.8)  to[bend  right] (0 , 0.7);
\draw (0,0.71) node[left] {$\gamma_{N-2}$};

\draw (0,0) node{$\bullet$};
\draw[->,>=latex] (0,0)  to[bend  left] (0 , 0.1);
\draw (0,0.1) node[left] {$\gamma_{0}$};

\end{tikzpicture}

\end{center}
\caption{The transition rates of the jump Markov process $(J_z,L_z)_{z\geq 0}$ on the state space $D_N=\{0 \leq j\leq l \leq N-1\}$. 
The absorption is concentrated on the line $(j,N-1)_{j=0}^{N-1}$.}
\label{fig:markov2}
\end{figure}
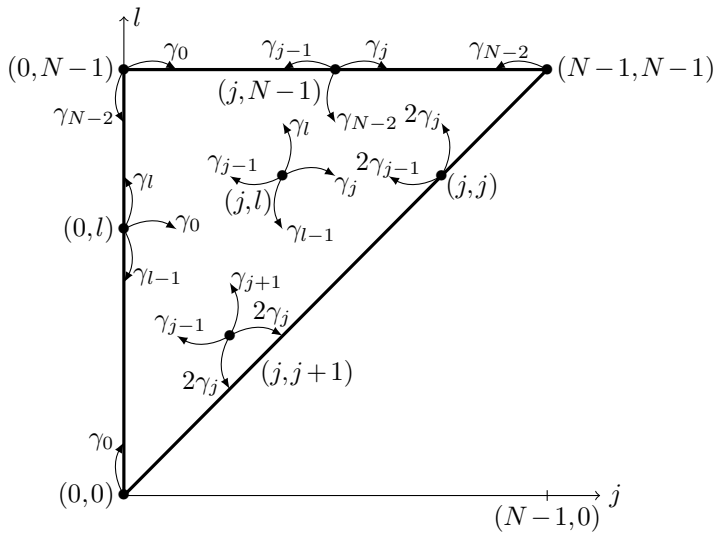

We proceed as in \cite{gomez} to determine the asymptotic behavior of $ S_{jl}$ when $N \to +\infty$.
We denote $ S_{jl}$ by $S_{jl}^{(N)}$ to keep track of the dependence with respect to $N$.
We introduce the triangle  $D=\{ (u,v )\in \RR^2, \, 0 < u < v< 1\}$.
We get that, for any function $\psi$ in $L^2(D)$, the function 
$$
S^{(N)}_\psi (z,u,v) = \sum_{0\leq j\leq l\leq N-1}  S^{(N)}_{jl} (z) \psi(\lfloor j/N\rfloor,\lfloor l/N\rfloor) ,
$$
where $S^{(N)}_{jl}$ is the solution of (\ref{eq:SNN}) starting from $S^{(N)}_{jl}(z=0)=\delta_{j , \lfloor u N \rfloor}\delta_{l , \lfloor v N \rfloor}$,
satisfies
$$
\lim_{N \to \infty} S^{(N)}_\psi (z,u,v)  =S_\psi(z,u,v)  ,
$$
where $S_\psi$ is the solution of the partial differential equation:
\begin{equation}
\partial_z S_\psi = {\cal H}_2 S_\psi, \quad {\cal H}_2  = \partial_u \big( \gamma_\infty(u)\partial_u \cdot)+ \partial_v \big( \gamma_\infty(v)\partial_v \cdot)  ,
\end{equation}
with the boundary condition (Dirichlet on the face $\{u=1\}$ of the triangle $D$,
Neumann on the faces $\{v=0\}$ and $\{u=v\}$, see Figure \ref{fig:markov3}):
\begin{equation}
\label{eq:bc2}
\partial_u S_\psi(z,0,v)=0, \quad
S_\psi(z,u,1)=0, \quad
(\partial_u -\partial_v) S_\psi(z,u,v) \mid_{u=v}=0,
\end{equation}
and the initial condition $S_\psi(0,u,v)=\psi(u,v)$.

\begin{figure}
\begin{center}
\begin{tikzpicture}[scale=5]
\draw [->] (0,0) --  (0,0.9 );
\draw (0,0.9) node[right] {$v$};
\draw [->] (0,0) --  (0.9,0) ;
\draw (0.9,0) node[right] {$u$};
\draw (0,0) node[left] {$(0,0)$};
\draw (0,0.8) node[left] {$(0,1)$};
\draw (0.8,-0.01) -- (0.8,0.01);
\draw (0.8,0) node[below] {$(1,0)$};
\draw [very thick] (0,0) --  (0,0.8) -- (0.8,0.8) -- (0,0) ;
\draw (0.8,0.8) node{$\bullet$};
\draw (0.8,0.8) node[right] {$(1,1)$};

\draw (0.4,0.4) node[right] {$Neumann$};
\draw (0,0.4) node[left] {$Neumann$};
\draw (0.4,0.8) node[above] {$Dirichlet$};

\end{tikzpicture}
\end{center}
\caption{The domain of the continuum approximation with its boundary conditions.}
\label{fig:markov3}
\end{figure}
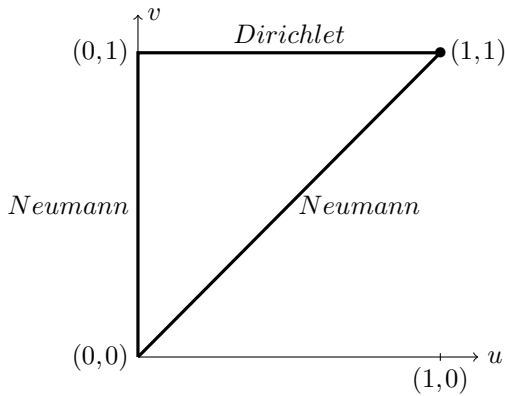

Consequently, we get the following result.

\begin{proposition}
\label{prop:mucont}
In the continuum approximation, 
the first eigenvalue $-\mu^{(N)}$ of $\boldsymbol{\Theta}^{(N)}-\boldsymbol{\Psi}^{(N)}$
converges as $N \to \infty$ to $-\mu$ with 
\begin{equation}
\label{def:mupropcont}
\mu = \inf_{\psi \in {\cal D}_2}
\iint_D \gamma_\infty(u) [\partial_u \psi(u,v)]^2 +\gamma_\infty(v) [\partial_v \psi(u,v)]^2 du dv  
\end{equation}
and
\begin{align}
\nonumber
{\cal D}_2= &\Big\{ \psi \in {\cal C}^\infty(\bar{D}) , \, \iint_D \psi(u,v)^2 du dv=1, \\
& \partial_u \psi (0,v)=0 ,
 \,  \psi(u,1)=0, 
 \, (\partial_u -\partial_v )\psi(u,v) \mid_{u=v}=0 \Big\} .
\end{align}
\end{proposition}
Equivalently,
\begin{equation}
\label{eq:caracmutilde2}
\mu = \inf_{\check{\psi} \in \check{\cal D}_2}
\iint_{(0,1)^2}  \gamma_\infty(u) [\partial_u \check{\psi}(u,v)]^2+\gamma_\infty(v) [\partial_v \check{\psi}(u,v)]^2 du dv  ,
\end{equation}
with
\begin{align}
\nonumber
\check{\cal D}_2=& \Big\{ \check{\psi} \in {\cal C}^\infty([0,1]^2) , \, \iint_{(0,1)^2} \check{\psi}(u,v)^2 du dv=1, \\
& \partial_u \check{\psi} (0,v)=0 , \, \partial_v \check{\psi} (u,0)=0 ,
 \,  \check{\psi}(u,1)=0, \, \check{\psi}(1,v)=0 \Big\} .
\end{align}
{\it Proof of (\ref{eq:caracmutilde2}).}
We denote by $\check{\mu}$ the right-hand side of (\ref{eq:caracmutilde2}).
We can show as in \cite{gomez} that $-\check{\mu}$ is a simple eigenvalue of the operator ${\cal H}_2$ on $(0,1)^2$ with  Dirichlet  boundary conditions on $\{v=1\}$ and $\{u=1\}$ and Neumann boundary conditions on $\{u=0\}$ and $\{v=0\}$ 
and that the corresponding eigenvector $\check{\psi}$ is smooth and unique up to a multiplication by $-1$. 
Moreover, the function  $\check{\psi}_2: (u,v) \in (0,1)^2 \mapsto   \check{\psi}(v,u)$ also satisfies ${\cal H}_2 \check{\psi}_2=- \check{\mu} \check{\psi}_2$ with the same Dirichlet/Neumann boundary conditions. By uniqueness, we get that $\check{\psi}$ is symmetric: $\check{\psi}(u,v)=\check{\psi}(v,u)$,
so it satisfies $(\partial_u -\partial_v) \check{\psi}(u,v) \mid_{u=v}=0$. Therefore 
$\psi_2: (u,v) \in D \mapsto \sqrt{2} \check{\psi}(u,v)$ is an eigenvector of ${\cal H}_2$ on $D$ with the boundary conditions
$\partial_u  \psi_2(0,v)=0$, $\psi_2(u,1)=0$, $(\partial_u -\partial_v) \psi_2(u,v) \mid_{u=v}=0$ with the eigenvalue $-\check{\mu}$.
If we use $\psi_2$ as a test function in (\ref{def:mupropcont}), we find that $\mu  \leq \check{\mu} $.\\
We can show as in \cite{gomez} that $-{\mu}$ is a simple eigenvalue of the operator ${\cal H}_2$ on $D$ with the 
Dirichlet/Neumann boundary conditions $\partial_u  \psi(0,v)=0$, $\psi(u,1)=0$, $(\partial_u -\partial_v) \psi(u,v) \mid_{u=v}=0$,
and that the corresponding eigenvector ${\psi}$ is smooth.
If we use $\check{\psi}(u,v) = \psi(u,v) {\bf 1}_{u \leq v}+ \psi(v,u){\bf 1}_{u > v}$  as a test function in (\ref{eq:caracmutilde2}), then we find that 
$\check{\mu}\leq \mu$.
\qed

Propositions \ref{prop:lambdacont} and \ref{prop:mucont} make it possible to prove the following identity
that establishes a simple relation between the growth rates of the means and variances of the mode powers in the continuum approximation.

%We know that we must have $\mu \geq 2 \lambda$
%because $\mu^{(N)}\geq 2 \lambda^{(N)}$ for any $N$,
%therefore we get the following result.

\begin{proposition}
In the continuum approximation, we have
\begin{equation}
\mu=2\lambda ,
\end{equation}
where $\lambda$ and $\mu$ are defined in Propositions \ref{prop:lambdacont} and \ref{prop:mucont}, respectively.
\end{proposition}

\begin{proof}
If ${\varphi}$ is the eigenvector of ${\cal H}_1$ with the boundary conditions $\varphi'(0)=0$, $\varphi(1)=0$ with eigenvalue $-\lambda$, 
then $\check{\psi}:(u,v) \in (0,1)^2 \mapsto {\varphi}(u) {\varphi}(v)$ is an eigenvector of
${\cal H}_2$ on $(0,1)^2$ with the boundary conditions $\partial_u \check{\psi} (0,v)=0 , \, \partial_v \check{\psi} (u,0)=0 ,
 \,  \check{\psi}(u,1)=0, \, \check{\psi}(1,v)=0$
with the eigenvalue $-2\lambda$. If we use $\check{\psi}$ as a test function in (\ref{eq:caracmutilde2}), then we find that $\mu \leq 2\lambda$. 

{The operator ${\cal H}_1$ is self-adjoint in $L^2(0,1)$ with Neumann boundary condition at $0$ and Dirichlet boundary condition at $1$. Therefore, there exists an eigenbasis $(\varphi_n)_{n\geq 0}$ with the eigenvalues $(-\lambda_n)_{n\geq 0}$, with $0 <\lambda_0 < \lambda_{1} \leq \cdots \leq \lambda_n \leq \cdots$.
The function $\varphi_0$ is the unique eigenvector 
of ${\cal H}_1$ associated to the eigenvalue $-\lambda_0=-\lambda$. The family of functions $(\check\psi_{m,n})_{m,n\geq 0}$ with $\check\psi_{m,n}(u,v)=\varphi_m(u)\varphi_n(v)$ forms a basis of the space $L^2((0,1)^2)$ with Neumann boundary conditions at $\{u=0\}$ and $\{v=0\}$ and Dirichlet boundary conditions at $\{u=1\}$ and $\{v=1\}$.
The function $\check\psi_{m,n}$ is an  eigenfunction of the operator ${\cal H}_2$, with the eigenvalue $-\mu_{m,n}=- \lambda_m -\lambda_n$.
Therefore, for any function $\check\psi \in \check{\cal D}_2$, we have $\check\psi = \sum_{m,n} c_{m,n} \check\psi_{m,n}$ with
$\sum_{m,n} c_{m,n}^2=1$ and ${\cal H}_2 \check\psi = - \sum_{m,n} c_{m,n} \mu_{m,n}\check\psi_{m,n}$,
so that 
\begin{align*}
\iint_{(0,1)^2}  \gamma_\infty(u) [\partial_u \check{\psi}(u,v)]^2+\gamma_\infty(v) [\partial_v \check{\psi}(u,v)]^2 du dv  
=
- \iint_{(0,1)^2} \check{\psi} {\cal H}_2 \check{\psi} (u,v) du dv\\
=
\sum_{m,n}  \mu_{m,n} c_{m,n}^2 \geq  2\lambda ,
\end{align*}
which shows that $\mu \geq 2\lambda$.
}
\end{proof}

By uniqueness this implies that the eigenvector $\check{\psi}$ of ${\cal H}_2$ on $(0,1)^2$ associated to $-\mu$ 
is $\check{\psi}:(u,v) \in (0,1)^2 \mapsto {\varphi}(u) {\varphi}(v)$.
This in turn implies that the eigenvector ${\psi}$ of ${\cal H}_2$ on $D$ associated to $-\mu=-2\lambda$ 
is $ {\psi}:(u,v) \in D \mapsto \sqrt{2} {\varphi}(u) {\varphi}(v)$.
As a result we get
$$
S_{jl}(z) \stackrel{z \to +\infty}{\simeq}
2 c_W \varphi(j/{N}) \varphi(l/N) \exp\big( -2 \lambda z \big)  ,
$$
with $c_W = \sum_{j,l=0}^{N-1} \varphi(j/N) \varphi(l/N) |a_{j,{\rm s}}|^2|a_{l,{\rm s}}|^2 /N^2= c_V^2$, $c_V = 
\sum_{j=0}^{N-1} \varphi(j/N)   |a_{j,{\rm s}}|^2|/N$,
and therefore
\begin{equation}
\label{eq:limitRcont}
R_{jl}(z) \stackrel{z \to +\infty}{\simeq}
c_V^2 (1+\delta_{jl})  \varphi(j/{N}) \varphi(l/N)\exp\big( -2 \lambda  z \big)  .
\end{equation}
This result is the key to show that we will not observe any exponential growth of the relative intensity fluctuations
in the continuum approximation.

{\bf Exponential growth of the intensity fluctuations.}
It is a general feature that, for any matrix $\boldsymbol{\Gamma}$ and effective dissipation coefficients
$\Lambda_j$, we have $\mu - 2\lambda  \leq 0$ 
(this is a consequence of Cauchy-Schwarz inequality: the square of the mean mode power cannot grow faster than the mean square mode power).
The first two moments of the pointwise intensity 
$|{p}(x,z)|^2$ for large $z$ are
\begin{align}
\EE [ |{p}(x,z)|^2] \stackrel{z \to \infty}{\simeq} 
& \sum_{j=0}^{N-1}\frac{\phi_j(x)^2}{\beta_j} c_V V_j e^{- \lambda z} ,\\
\EE [ |{p}(x,z)|^4 ] \stackrel{z \to \infty}{\simeq} 
& \sum_{j,l=0}^{N-1} \frac{\phi_j(x)^2\phi_l(x)^2}{\beta_j \beta_l} c_W W_{jl} e^{- \mu z} .
\end{align}

Without dissipation we have the following result for the relative fluctuations of the pointwise intensity:
$$ 
\frac{\EE[|{p}(x,z)|^4]}{\EE[|{p}(x,z)|^2]^2} \stackrel{z \to \infty}{\longrightarrow} 
\frac{2N}{N+1}  ,
$$
which is equal to $2$ when $N \gg 1$.

With dissipation
\begin{equation}
\frac{\EE[|{p}(x,z)|^4]}{\EE[|{p}(x,z)|^2]^2} \stackrel{z \to \infty}{\sim} 
\exp\big( - (\mu - 2\lambda) z \big)  ,
\label{eq:si1}
\end{equation}
that grows exponentially with the propagation distance.
With weak dissipation,
\begin{equation}
\frac{\EE[|{p}(x,z)|^4]}{\EE[|{p}(x,z)|^2]^2} \stackrel{z \to \infty}{\simeq} 
\frac{2N}{N+1} 
\exp\big( - (\mu - 2\lambda) z \big) \big(1 +o(1) \big) ,
\label{eq:si1b}
\end{equation}
because the first eigenvectors ${\itbf V}$ and ${\itbf W}$ are close to the ones
of the non-dissipative case.
Note, however, that the exponential growth happens only 
for very long distances, because $|\mu - 2\lambda |$
is very small as shown above.
Eq.~(\ref{eq:diffvp}) gives the expression of the exponential growth rate 
when dissipation is weak and $\Gamma_{jl}\equiv \Gamma$ for $j \neq l$: the growth rate increases when the effective modal dissipation coefficients
become different from each other and decreases when the number of modes increases.
The analysis in the continuum approximation confirms that the exponential growth rate of the relative intensity fluctuations vanishes when
the number of modes goes to infinity. 
More exactly, in the continuum approximation, when the number of modes becomes large, we have $\mu=2\lambda$ and (\ref{eq:limitRcont}) holds. 
Therefore there is no exponential growth of the fluctuations and we have 
\begin{equation}
\frac{\EE[|{p}(x,z)|^4]}{\EE[|{p}(x,z)|^2]^2} \stackrel{z \to \infty}{\simeq} 
2 ,
\end{equation}
which corresponds to a relative variance (or scintillation index) equal to one.
We recover the standard result that the wavefield, in the point of view of the fourth-order moments,
behaves as a Gaussian process with relative variance (scintillation index) equal to one~\cite{garniersolna:arma}.

\section{Conclusion}
In this paper we have reviewed the asymptotic theory of wave propagation in random waveguides.
We have recovered standard results about the first two moments of the mode amplitudes: the mean amplitudes decay exponentially and
the mean powers satisfay a coupled mode equation.
The fourth-order moment analysis also reveals that the fluctuation of the mode powers grow exponentially with the propagation distance.
We have carefully studied the exponential growth rates of the relative variances.
We have shown that, when the number of guided modes increases, the exponential growth rates vanish
and the scintillation index (the relative variance of the intensity fluctuations) becomes equal to one,
as observed in open medium in the random paraxial regime \cite{garniersolna:arma}.
{These results show that incoherent imaging in a random waveguide (such as a Pekeris waveguide in underwater acoustics)
is challenging. Indeed incoherent imaging is based on the use of the cross correlations of the recorded signals \cite{dumaz}.
The estimation of the second-order moments of the wavefield is, however, extremely difficult because of the large variances of the empirical second-order moments and one may need to average over a lot a samples (while the medium may be not stationary as in underwater acoustics).
This is in contrast with the situation in open three-dimensional random media where smoothed Wigner transforms are statistically stable \cite{bal,garniersolna:arma}.
More generally, the results on the fourth-order moments show that the predictions of the coupled mode equations
(which describe the evolutions of the statistical second-order moments of the wavefield, such as Eq.~(\ref{eq:momQ})) 
are not easy to exploit experimentally when the number of guided modes is not very large.}

\appendix

\section{Wave mode decomposition}
\label{app:dec}
Let us  introduce the Helmholtz operator
\begin{align}
{\cal H} = 
 \partial_x^2 + k^2 {\rm n}^{(0)}(x)^2  .
 \label{eq:helmholtz}
\end{align}
The Helmholtz operator ${\cal H}$ is self-adjoint with respect to the standard scalar product defined on $L^2(\RR)$ by:
\begin{align}
\nonumber
(\phi_1,\phi_2)_{L^2} := \int_\RR  \overline{\phi_1(x)} {\phi_2(x)}dx .
\label{scalarproduct}
\end{align}
The Helmholtz operator has a spectrum of the form (\ref{eq:spectrum0})
where the $N$ modal wavenumbers $\beta_j$ are positive and 
$k^2< \beta_{N-1}^2 < \cdots <\beta_0^2 <n^2 k^2$.

{\bf Discrete spectrum.}
The $j$th eigenvector associated to the eigenvalue $\beta_j^2$ is even for even $j$:
\begin{align}
\phi_j(x) = 
\left\{
\begin{array}{l}
A_j \cos (\sigma_j x/d) \mbox{ if } 0\leq |x| \leq d/2\\
A_j \cos (\sigma_j/2 ) \exp(-\zeta_j (|x|/d-1/2))\mbox{ if }  |x| \geq d/2
\end{array}
\right.
\end{align}
and odd for odd $j$
\begin{align}
\phi_j(x) = 
\left\{
\begin{array}{l}
A_j \sin (\sigma_j x/d) \mbox{ if } 0\leq |x| \leq d/2\\
A_j \sin (\sigma_j/2 ) {\rm sgn}(x) \exp(-\zeta_j (|x|/d-1/2))\mbox{ if }  |x| \geq d/2
\end{array}
\right.
\end{align}
where 
\begin{align}
\sigma_j =   \sqrt{n^2 k^2  -\beta_j^2} d, \quad \zeta_j =  \sqrt{\beta_j^2-{k}^2 } d ,
\end{align}
and
\begin{align}
A_j^2 =
\left\{
\begin{array}{ll}
 \frac{1/d}{ ( \frac{1}{2}+\frac{\sin (\sigma_j)}{2\sigma_j})
+  \frac{\cos^2 (\sigma_j/2)}{\zeta_j} }  & \mbox{ for even $j$}
\\
 \frac{1/d}{ ( \frac{1}{2}-\frac{\sin (\sigma_j)}{2\sigma_j})
+  \frac{\sin^2 (\sigma_j/2)}{\zeta_j} }  & \mbox{ for odd $j$}
\end{array}
\right.
\end{align}
For even $j$ the $\sigma_j$'s are the solutions in $(0, \sqrt{n^2-1}kd)$ of
\begin{align}
 \tan (\sigma/2) = \frac{ \sqrt{  (n^2-1) k^2d^2 -\sigma^2} }{\sigma}.
\end{align}
For odd $j$ the $\sigma_j$'s are the solutions in $(0, \sqrt{n^2 -1}kd)$ of
\begin{align}
\tan (\sigma/2) = - \frac{\sigma}{\sqrt{ (n^2-1) k^2d^2 -\sigma^2}}   ,
\end{align}
and we denote by $N$ the number of solutions.
We have $\sigma_{j} \in (j\pi,(j+1)\pi)$
and 
\begin{equation}
\label{eq:numberofmodes}
N=\lfloor \sqrt{n^2-1} k d/\pi \rfloor.
\end{equation}

{\bf Continuous spectrum.}
For $\gamma \in (-\infty,{k}^2)$, there are two improper eigenvectors (one is even and the other one is odd) and they have the form:
\begin{align}
\phi_{e,\gamma}(x) = 
\left\{
\begin{array}{l}
A_{e,\gamma} \cos (\eta_\gamma x/d) \mbox{ if } 0\leq |x|\leq d/2\\
A_{e,\gamma} \big[ \cos (\eta_\gamma /2) \cos(\xi_\gamma (|x|/d-1/2)) \\
\, \, \,
- {\eta_\gamma}/{\xi_\gamma} \sin(\eta_\gamma/2) \sin(\xi_\gamma (|x|/d-1/2))
\big]
\mbox{ if }  |x| \geq  d/2
\end{array}
\right.
\end{align}
\begin{align}
\phi_{o,\gamma}(x) = 
\left\{
\begin{array}{l}
A_{o,\gamma} \sin (\eta_\gamma x/d) \mbox{ if } 0\leq |x|\leq d/2\\
A_{o,\gamma} {\rm sgn}(x) \big[ \sin (\eta_\gamma /2)  \cos(\xi_\gamma (|x|/d-1/2)) \\
\, \, \,
+ {\eta_\gamma}/{\xi_\gamma} \cos(\eta_\gamma/2) \sin(\xi_\gamma (|x|/d-1/2))
\big]
\mbox{ if }  |x| \geq  d/2
\end{array}
\right.
\end{align}
where
\begin{align}
\eta_\gamma= \sqrt{n^2 k^2-\gamma}  d ,\quad \xi_\gamma = \sqrt{{k}^2-\gamma} d  ,
\end{align}
and
\begin{align}
A_{e,\gamma}^2 =& \frac{\xi_\gamma d}{2 \pi (\xi_\gamma^2 \cos^2(\eta_\gamma/2) +
  \eta_\gamma^2 \sin^2(\eta_\gamma/2))} ,\\
A_{o,\gamma}^2 =& \frac{\xi_\gamma d}{2 \pi (\xi_\gamma^2 \sin^2(\eta_\gamma/2) +
 \eta_\gamma^2 \cos^2(\eta_\gamma/2))}  .
\end{align}
We remark that $\phi_{t,\gamma}$ does not belong to $L^2(\RR)$, but 
$\left( \phi_{t,\gamma},\phi\right)_{L^2}$ can be defined for any test function $\phi \in L^2(\RR)$ as
\begin{align}
\left( \phi_{t,\gamma},\phi\right)_{L^2} = \lim_{M \to +\infty} \int_{-M}^M \phi_{t,\gamma}(x) \phi(x)  dx
,
\end{align}
where the limit holds (as a function in $\gamma$) in $L^2((-\infty,{k}^2))$.

{\bf Completeness.}
We have for any $\phi\in L^2(\RR)$:
\begin{align}
\left( \phi,\phi \right)_{L^2}  = \sum_{j=0}^{N-1} \big| \left( \phi_j ,\phi\right)_{L^2} \big|^2
+
\sum_{t\in \{e,o\}} \int_{-\infty}^{{k}^2} \big| \left( \phi_{t,\gamma},\phi\right)_{L^2}\big|^2 d\gamma   .
\end{align}
The map which assigns to every element of $\phi \in L^2(\RR)$ the coefficients of its spectral decomposition
$$
\phi \mapsto \Big( \left( \phi_j,\phi\right)_{L^2}  ,j=0,\ldots,N-1, 
 \left(\phi_{t,\gamma},\phi\right)_{L^2}, t\in \{e,o\},\gamma\in ({-\infty},{{k}^2})\Big)
$$
is an isometry from $L^2(\RR)$ onto $\CC^N \times L^2((-\infty,{k}^2))^2$.
This means that any function $\phi\in L^2(\RR)$ can be expanded on the set of the eigenfunctions of ${\cal H}$.


\begin{thebibliography}{99}

\bibitem{alonso}
R. Alonso, L. Borcea, and J. Garnier, 
Wave propagation in waveguides with rough boundaries, 
Commun. Math. Sci. {\bf 11}, 233--267 (2012).

%\bibitem{bal}
%G. Bal, T. Komorowski, and L. Ryzhik,
%Self-averaging of Wigner transforms in random media,
%Commun. Math. Phys. {\bf 242}, 81--135  (2003).

\bibitem{bal}
{
G. Bal and O. Pinaud,
Self-averaging of kinetic models for waves in random media,
Kinetic Related Models {\bf 1}, 85--100 (2008).
}

\bibitem{beran}
M. J. Beran and S. Frankenthal,
Volume scattering in a shallow channel,
J. Acoust. Soc. Am. {\bf 91}, 3203--3211 (1992).

\bibitem{borcea15}
L. Borcea, J. Garnier, and C. Tsogka, 
A quantitative study of source imaging in random waveguides, 
Commun. Math. Sci. {\bf 13}, 749--776 (2015).

\bibitem{colosi12}
J. A. Colosi, T. F. Duda, and A. K. Morozov,
Statistics of low-frequency normal-mode amplitudes in an ocean with random sound-speed perturbations: 
Shallow-water environments,
J. Acoust. Soc. Am. {\bf 131}, 1749--1761 (2012).

\bibitem{colosi09}
J. A. Colosi and A. Morozov,  
Statistics of normal mode amplitudes in an ocean with random sound speed perturbations: 
Cross mode coherence and mean intensity,
J. Acoust. Soc. Am. {\bf 126}, 1026--1035 (2009).

\bibitem{creamer}
D. Creamer, 
Scintillating shallow water waveguides,
J. Acoust. Soc. Am. {\bf 99}, 2825--2838 (1996).

\bibitem{dozier}
L. B.  Dozier and F. D.  Tappert, 
Statistics of normal-mode amplitudes in a random ocean. I. Theory, 
J. Acoust. Soc. Am. {\bf 63}, 353--365 (1978). 

\bibitem{dumaz}
{L. Dumaz, J. Garnier, and G. Lepoultier,
Acoustic and geoacoustic inverse problems in randomly perturbed shallow-water environments,
J. Acoust. Soc. Am. {\bf 146}, 458--469 (2019).}

\bibitem{FGPSbook}
J.-P. Fouque, J. Garnier, G. Papanicolaou, and K. S\o lna,
\emph{Wave Propagation and Time Reversal in Randomly Layered Media},
Springer, New York, 2007

\bibitem{garnier_evan}
J. Garnier, 
The role of evanescent modes in randomly perturbed single-mode waveguides,
Discrete and Continuous Dynamical Systems B {\bf 8}, 455--472  (2007). 

\bibitem{GP07}
J. Garnier and G. Papanicolaou, 
Pulse propagation and time reversal in random waveguides,
SIAM J. Appl.  Math. {\bf 67},  1718--1739 (2007).

\bibitem{garniersolna:wrcm}
J. Garnier and K. S\o lna, 
On effective attenuation in multiscale composite media, 
Waves in Random and Complex Media {\bf 25}, 482--505 (2015).

\bibitem{garniersolna:arma}
J. Garnier and K. S\o lna, 
Fourth-moment analysis for beam propagation in the white-noise paraxial regime,
Arch. Rational Mech. Anal.  {\bf 220}, 37--81 (2016).

\bibitem{gomez}
C. Gomez,
Wave propagation in shallow-water acoustic random waveguides,
Commun. Math. Sci. {\bf 9},  81--125 (2011).

\bibitem{gomez2}
C. Gomez,
Wave propagation in underwater acoustic waveguides with rough boundaries,
Commun. Math. Sci. {\bf 13}, 2005--2052 (2015).

\bibitem{kohler77}
W. Kohler and G. Papanicolaou,
Wave propagation in  randomly inhomogeneous ocean,
in Lecture Notes in Physics, Vol. 70,
J. B. Keller and J. S. Papadakis, eds.,
Wave Propagation and Underwater Acoustics,
Springer Verlag, Berlin, 1977.

\bibitem{magnanini}
R. Magnanini  and F. Santosa,
Wave propagation in a 2-d optical waveguide,
SIAM J. Appl. Math. {\bf 61}, 1237--1252 (2000).

\bibitem{marcuse69}
D. Marcuse,
{Radiation losses of dielectric waveguides in terms of the power spectrum of the wall distortion function},
Bell System Technical Journal {\bf 48}, 3233--3242 (1969).

\bibitem{marcuse}
D. Marcuse,
{\it Theory of Dielectric Optical Waveguides}, 
Academic Press, New York, 1974.

\bibitem{papa74}
G. Papanicolaou and W. Kohler, 
{Asymptotic theory of mixing stochastic ordinary differential equations}, 
Commun. Pure Appl. Math. {\bf 27}, 641--668 (1974).

\bibitem{sv}
D. W. Stroock and S. R. S. Varadhan, 
{\it Multidimensional Diffusion Processes}, 
Springer-Verlag, Berlin, 1979.

\bibitem{wilcox}
C. Wilcox, 
Spectral analysis of the Pekeris operator in the theory of acoustic wave propagation in shallow water, 
Arch. Rational Mech. Anal. {\bf 60}, 259--300 (1976).

\end{thebibliography}
\end{document}